\documentclass[a4paper,12pt, oneside]{amsart}
\usepackage{graphicx,amssymb,amstext,amsmath,amsfonts,amsthm}
\usepackage{latexsym}
\usepackage{bbm}
\usepackage{xypic}
\usepackage{nicefrac}
\usepackage{pstricks-add}
\usepackage{setspace}
\usepackage{atbegshi}
\usepackage{anysize}
\marginsize{3cm}{2cm}{3cm}{2cm}
\usepackage{setspace}
\newtheorem{theorem}{Theorem}[section]
\newtheorem{lemma}[theorem]{Lemma}
\newtheorem{proposition}[theorem]{Proposition}

\newtheorem{corollary}[theorem]{Corollary}

\theoremstyle{definition}
\newtheorem{definition}[theorem]{Definition}

\theoremstyle{remark}
\newtheorem{remark}[theorem]{Remark}

\numberwithin{equation}{section}
\newcommand\norm[1]{\left\lVert#1\right\rVert_{\mathbb{TV}}}
\newcommand{\N}{\mathcal{N}}
\newcommand{\R}{\mathbb{R}}
\newcommand{\x}{{x^{\epsilon}_t}}

\newcommand{\y}{{y^{\epsilon}_t}}

\author{Barrera, G.}
\address{
Center for Research in Mathematics, CIMAT. Jalisco S/N, Valenciana, ZIP: $36240$ Guanajuato, Guanajuato, Mexico.}
\email{bvargas@cimat.mx }
\address{Instituto de Matem\'atica Pura e Aplicada, IMPA. Estrada Dona Castorina $110$, ZIP:$22460$-$320$,
Rio de Janeiro, Rio de Janeiro, Brasil.}
\email{bvargas@impa.br}
\author{Jara, M.}
\address{Instituto de Matem\'atica Pura e Aplicada, IMPA. Estrada Dona Castorina $110$, ZIP:$22460$-$320$,
Rio de Janeiro, Rio de Janeiro, Brasil.}
\email{monets@impa.br}
\thanks{Research supported by a grant from CNPq}
\title{{Abrupt Convergence for Stochastic Small Perturbations of One Dimensional Dynamical Systems}}
\keywords{Cut-off Phenomenon, Total Variation Distance, Brownian Motion, Perturbed Dynamical Systems, Stochastic Differential Equations.}
\date{\today}
\begin{document}
\maketitle

\begin{abstract}
We study the cut-off phenomenon for a family of stochastic small perturbations of a one dimensional dynamical system. We will  focus in a semiflow of a deterministic differential equation which is perturbed by adding to the dynamics a white noise of small variance. Under suitable hypothesis on the potential we will prove that the family of perturbed stochastic differential equations present a profile cut-off phenomenon with respect to the total variation distance. We also prove a local cut-off phenomenon in a neighborhood of the local minima (metastable states) of multi-well potential.
\end{abstract}
\markboth{Stochastic Small Perturbations of Dynamical Systems}{Stochastic Small Perturbations of Dynamical Systems}
\section{Introduction}

In the last decades intense research has been devoted to the study of dynamical systems subjected to random perturbations. Considerable effort has been dedicated to investigate exit times and exit locations from given domains and how they relate to the respective deterministic dynamical system. The theory of large deviations provides the usual mathematical framework for tackling these problems in case of Gaussian perturbations. This theory sets up the precise time scales for transitions of non degenerate stochastic systems between certain regimes. The theory of random dynamical systems, on the other hand, assigns Lyapunov exponents to linear random dynamical systems. These are the exponential growth rates as time grows large for fixed intensities of the underlying noise. For details see M. Freidling \& A. Wentzell \cite{FW}, \cite{FW1}, \cite{FW2}, M. Day \cite{MD}, \cite{MD1} and W. Siegert \cite{SI}. 
We will study the relation to the respective deterministic dynamical systems from a different point of view.

We study the asymptotically behavior or the so-called {\em cut-off phenomenon} for a family of stochastic small perturbations of a given dynamical system.
We will focus on the semiflow of a deterministic differential equation which is perturbed by adding to the dynamics a white noise perturbations.
Under suitable hypotheses on the vector field (coercivity assumption) we will prove that the
one parameter family of perturbed stochastic differential equations presents a profile cut-off in the sense of the definition of cut-off given by
J. Barrera \& B. Ycart \cite{BY}.

The term ``cut-off''   was introduced by D. Aldous and P. Diaconis in \cite{AD} in the early eighties to describe the phenomenon of abrupt convergence of Markov chains introduced as models of shuffling cards. Since the appearance of \cite{AD} many families of stochastic processes have been shown to have similiar properties. For a good introduction to the  different definitions of cut-off and the evolution of the concept in discrete time, see
J. Barrera \& B. Ycart
\cite{BY} and
P. Diaconis
\cite{DI}. In \cite{SA}, L. Saloff-Coste gives an extensive list of random walks for which the phenomenon occurs. Now, it us a well studied feature of Markov processes.

\textit{What does the ``cut-off'' phenomenon mean?}
It refers to an asymptotically drastic convergence of a family of stochastic processes labeled by some parameter.
Before a certain ``cut-off time'' those processes stay far from equilibrium in the sense that a suitable distance in some sense between the distribution at time $t$ and the equilibrium measure is far from $0$;
after a deterministic time  ``the cut-off time'' the distance decays exponentially fast to zero.

The term ``cut-off''  is naturally associated to such an ``all/nothing''  or ``1/0 behavior'', but it has the drawback of being used with other meanings in statistical mechanics and theoretical physics.
Alternative names have been proposed, including {\em threshold phenomenon} and {\em abrupt convergence}.

Consider a one parameter family of stochastic processes in continuous time $\{x^{\epsilon}\}_{\epsilon>0}$ indexed by $\epsilon>0$,
$x^{\epsilon}:=\{x^{\epsilon}_t\}_{t\geq 0}$, each one converging to a asymptotic distribution $\mu^{\epsilon}$ when $t$ goes to infinity. Let us denote by $d_{\epsilon}(t)$ the distance between the distribution at time $t$ of the $\epsilon$-th processes, $\mathbb{P}\left(x^\epsilon_t\in \cdot\right)$, and its asymptotic distribution as $t\rightarrow +\infty$, $\mu^\epsilon$, where the ``distance''  can be taken to being the total variation, separation, Hellinger, relative entropy, Wasserstein, $L^{p}$ distances, etc. Following J. Barrera \& B. Ycart \cite{BY}, the cut-off phenomenon for $\{x^\epsilon\}_{\epsilon >0}$ can be expressed at three increasingly sharp levels. Let us denoted by $M$ the diameter of the respective metric space of probability measures in which we are working. In general, $M$ could be infinite. In our case, we will focus on the total variation distance so $M=1$.
\begin{definition}[Cut-Off]\label{cutoff}
The family $\{x^{\epsilon}\}_{\epsilon>0}$ has a cut-off at $\{t_{\epsilon}\}_{\epsilon>0}$  if $t_{\epsilon} \rightarrow +\infty$ when $\epsilon \rightarrow 0$ and
\begin{eqnarray*}
\lim\limits_{\epsilon\rightarrow 0 }{d_{\epsilon}(ct_{\epsilon})}= \left\{ \begin{array}{lcc}
             M &   if  & 0 < c < 1, \\
             \\ 0 &  if & c>1. \\
             \end{array}
   \right.
\end{eqnarray*}
\end{definition}

\begin{definition}[Window Cut-Off]\label{windows}
The family $\{x^{\epsilon}\}_{\epsilon>0}$
 has a window cut-off at
$\{\left(t_{\epsilon}, w_{\epsilon}\right)\}_{\epsilon>0}$, if $t_{\epsilon} \rightarrow +\infty$ when $\epsilon \rightarrow 0$, $w_{\epsilon}=o\left(t_{\epsilon}\right)$ and
\begin{eqnarray*}
\lim\limits_{c \rightarrow -\infty}{\liminf\limits_{\epsilon\rightarrow 0}
{d_{\epsilon}(t_{\epsilon}+cw_{\epsilon})}}&=&M,\\
\lim\limits_{c \rightarrow +\infty}{\limsup\limits_{\epsilon\rightarrow 0}
{d_{\epsilon}(t_{\epsilon}+cw_{\epsilon})}}&=&0.\\
\end{eqnarray*}
\end{definition}

\begin{definition}[Profile Cut-Off]\label{profile}
The family $\{x^{\epsilon}\}_{\epsilon>0}$ has profile cut-off at
$\{\left(t_{\epsilon}, w_{\epsilon}\right)\}_{\epsilon>0}$ with profile $G$, if  $t_{\epsilon} \rightarrow +\infty$ when $\epsilon \rightarrow 0$, $w_{\epsilon}=o\left(t_{\epsilon}\right)$,
\begin{eqnarray*}
 G(c):=\lim\limits_{\epsilon \rightarrow 0}{d_{\epsilon}(t_{\epsilon}+cw_{\epsilon})}
\end{eqnarray*} exists for all $c\in \mathbb{R}$ and
\begin{eqnarray*}
\lim\limits_{c \rightarrow -\infty}{G(c)}&=&M,\\
\lim\limits_{c \rightarrow +\infty}{G(c)}&=&0.
\end{eqnarray*}
\end{definition}
We also give a mathematical description of the phenomenon of metastability when the potential is a double well potential with some smooth conditions and certain increase rate at infinity.
According to the initials conditions the deterministic trajectories associated to the differential equation (\ref{dde1}) converge to the local minima of the potential $V$ or stay in its local minima. 
Therefore, no transition between different domains of attraction is possible. This situation becomes different if we perturb the deterministic differential equation (\ref{dde1}) by a small aditive noise whose presence allows transitions between the potential wells. Depending on the initial conditions of the system and the properties of the noise certain potential wells may be reached only on appropriated long time scales or stay unvisited.

The phenomenon of metastability roughly speaking means that for different time scales and initial conditions the system may reach different local statistical equilibria.
Dynamical systems subject to small Gaussian perturbations have been studied extensively, for details see \cite{FW}.
The theory of large deviations allows to solve the exit problem from the domain of attraction of a stable point.
It turns out that the mean exit time is exponentially large in the small noise parameter, and its logarithmic rate is proportional to the height of the potential barrier the trajectories have to overcome.
Consequently, for a multi-well potential one can obtain a series of exponentially non-equivalent time scales given by the wells mean exit times. 
Moreover, one can prove that the normalised exit times are exponentially distributed and have a memoryless property. For details see \cite{FW} for Gaussian perturbations and \cite{II} for L\'evy-driven diffusions.

This material will be organized as follows. Section \ref{section1} describes the model and states the main results besides establishing the basic notation. Section \ref{section2} provides the results for a linear approximations which is an essential tool in order to obtain the main results. Section \ref{gradcase}
gives the ingredients in order to obtain the main results and provides the proof of the main results.
Section \ref{dwp} studies a kind of local cut-off phenomenon in a neighborhood of the local minima (metastable states) of multi-well potential.
The Appendix  is divided in three section as follows:
Section \ref{ap1} gives elementary properties for the total variation distances of Normal distributions. Section \ref{ap2} provides the proofs that we do not proof in Section \ref{section2} and Section 
\ref{gradcase} in order to the lecture be fluent.
Section \ref{ap3} gives some useful results that we use along of this material.

\section{Stochastic Perturbations: One Dimensional Case}\label{section1}
On this section, let $x_0 \in \mathbb{R} \setminus \{0\}$ be fixed and let us consider the semiflow $\{\psi_t\}_{t \geq 0}$ associated to the solution of the following deterministic differential equation,
\begin{eqnarray}\label{dde1}
dx_t&=&-V^{\prime}(x_t)dt
\end{eqnarray}
for $t \geq 0$. The hypothesis made in Theorem \ref{main} on the potential $V$ guarantees existence and uniqueness of solutions of \eqref{dde1}, as well as all the other (stochastic or deterministic) equations defined below.

Let us establish some basic notation.
Let us take $\mu\in \mathbb{R}$ and let  $\sigma^2\in ]0,+\infty[$ be  fixed numbers.
We denote by $\mathcal{N}{\left(\mu,\sigma^2\right)}$ the Normal distribution with mean $\mu$ and variance $\sigma^2$. Given two probability measures $\mathbb{P}$ and $\mathbb{Q}$ which are defined in the same measurable space $\left(\Omega,\mathcal{F}\right)$,
we denote the total variation distance between $\mathbb{P}$ and $\mathbb{Q}$ by
$$\norm{\mathbb{P}-\mathbb{Q}}:=\sup\limits_{A\in \mathcal{F}}{|\mu(A)-\nu(A)|}.$$ Along this paper we always consider $\epsilon>0$.

Our main Theorem in the one dimensional case is the following:

\begin{theorem}[General Case]\label{main}
Let $V:\mathbb{R}\rightarrow \mathbb{R}$ be an one dimensional potential
that satisfies the following:
\begin{itemize}
\item[$i)$] $V\in \mathcal{C}^3$.
\item[$ii)$] $V(0)=0$.
\item[$iii)$] $V^{\prime}(x)=0$ if only if $x=0$.
\item[$iv)$] There exists $\delta>0$ such that
$V^{\prime\prime}(x)\geq \delta$ for every $x\in \R$.
\end{itemize}
Let us consider the family of processes indexed by
$\epsilon > 0$, $x^{\epsilon}=\{ {x^{\epsilon}_t} \}_{t \geq 0}$
which are given by the the semiflow of the following stochastic differential equation,
\begin{eqnarray*}
d{x^{\epsilon}_t}&=&-V^{\prime}({x^{\epsilon}_t})dt+\sqrt{\epsilon}dW_t,\\
{x^{\epsilon}_0}&=&x_0
\end{eqnarray*}
for $t \geq 0$, where $x_0$ is a deterministic initial condition in $\mathbb{R}\setminus \{0\}$ and
$\{W_t\}_{t \geq 0}$ is a standard Brownian motion. This family
presents profile cut-off  in the sense of
the Definition \ref{profile}  with respect to the total variation distance when $\epsilon$ goes to zero.
The profile function $G:\mathbb{R}\rightarrow \mathbb{R}$ is given by
\begin{equation*}
G(b):=\norm{\N{\left(\tilde{c} e^{-b},1\right)}-\N{\left(0,1\right)}},
\end{equation*}
where $\tilde{c}$ is the nonzero constant given by
\begin{eqnarray*}
\lim\limits_{t\rightarrow +\infty}{e^{V^{\prime\prime}{(0)}t}\psi_t}&=:&\tilde{c}.
\end{eqnarray*}
The cut-off time $t_{\epsilon}$ and window time $w_{\epsilon}$ are given by
\begin{eqnarray*}
t_{\epsilon}&:=&\frac{1}{2V^{\prime\prime}(0)}\left[\ln\left(\frac{1}{\epsilon}\right)+\ln\left(2V^{\prime\prime}(0)\right) \right],\\
w_{\epsilon}&:=&\frac{1}{V^{\prime\prime}(0)}+\delta_\epsilon,
\end{eqnarray*}
where $\delta_\epsilon=\epsilon^{\gamma}$ for some $\gamma\in ]0,1[$.
\end{theorem}

This Theorem will be proved at the end of the section \ref{gradcase}.

\section{The Linearized Case}\label{section2}
As an important intermediate step, we prove profile cut-off for a family of processes satisfying a linear, non-homogeneous stochastic differential equation which we will define bellow.
This result holds for a more general class of potentials that Theorem \ref{main}, which we define as follows.
\begin{definition}
We say that $V$ is a regular potential if
$V:\mathbb{R}\rightarrow \mathbb{R}$ satisfies
\begin{itemize}
\item[$a)$] $V$ is $\mathcal{C}^3$.
\item[$b)$] $V(0)=0$.
\item[$c)$] $V^{\prime}(x)=0$ iff $x=0$.
\item[$d)$] $V^{\prime\prime}(0)>0$.
\item[$e)$] $\lim\limits_{|x|\rightarrow +\infty}{V(x)}=+\infty$.
\end{itemize}
\end{definition}

In order to prove Theorem \ref{main} we will prove the analogous result for a ``linear approximations'' of the potential $V$.

\begin{theorem}[The Linearized Case]\label{toy}
Let us consider the family of processes indexed by
$\epsilon > 0$, $y^{\epsilon}=\{ {y^{\epsilon}_t} \}_{t \geq 0}$
which are given by the solution of the following linear stochastic differential equation,
\begin{equation}
\label{linearapprox}
\begin{array}{rcl}
d{y^{\epsilon}_t}&=&-V^{\prime\prime}{(\psi_t)}{y^{\epsilon}_t}dt+\sqrt{\epsilon}dW_t,\\
{y^{\epsilon}_0}&=&y_0
\end{array}
\end{equation}for $t \geq 0$, where $y_0$ is a deterministic initial condition in $\R\setminus\{0\}$,
$\{W_t\}_{t \geq 0}$ is a standard Brownian motion, $V$ is a regular potential and $\{\psi_t\}_{t\geq 0}$ is given by the solution of the deterministic differential equation (\ref{dde1}).
This family
presents profile cut-off  in the sense of
the Definition \ref{profile}
 with respect to the total variation distance when $\epsilon$ goes to zero.
The profile function $G:\mathbb{R}\rightarrow \mathbb{R}$ is given by
\begin{eqnarray*}
G(b)&:=&\norm{\N{\left(c e^{-b},1\right)}-\N{\left(0,1\right)}},
\end{eqnarray*}
where $c$ is the nonzero constant given by
\begin{eqnarray*}
\lim\limits_{t\rightarrow +\infty}{e^{V^{\prime\prime}{(0)}t}\Phi_t}&=:&c,
\end{eqnarray*}
where $\Phi=\{\Phi_{t}\}_{t\geq 0}$ is the fundamental solution of the non autonomous system
\begin{eqnarray*}
d\Phi_{t}=-V^{\prime\prime}{(\psi_t)}\Phi_t dt
\end{eqnarray*}
for every $t\geq 0$ with initial condition $\Phi_{0}=1$.
The cut-off time $t_{\epsilon}$ and window time $w_{\epsilon}$ are given by
\begin{eqnarray*}
t_{\epsilon}&:=&\frac{1}{2V^{\prime\prime}(0)}\left[\ln\left(\frac{1}{\epsilon}\right)+\ln\left(2V^{\prime\prime}(0) y^2_0\right) \right],\\
w_{\epsilon}&:=&\frac{1}{V^{\prime\prime}(0)}.
\end{eqnarray*}
\end{theorem}

Notice that choosing $V(x) = \frac{\alpha x^2}{2}$ for every $x\in \mathbb{R}$, where $\alpha>0$ is a fixed constant, we see that the Ornstein-Uhlenbeck process presents profile cut-off. 

In order to prove Theorem \ref{toy}, we will find the qualitative behavior of the semiflow $\psi=\{\psi_t\}_{t\geq 0}$ at infinity.
The following lemma tells us the asymptotic behavior of the expectation and variance of the ``linear approximations''.

\begin{lemma}\label{asin}
Leu us assume the hypothesis of Theorem \ref{toy}.
Then 
\begin{itemize}
\item[$i)$]  $\lim\limits_{t\rightarrow +\infty}{\psi_t}=0$.
\item[$ii)$] $\lim\limits_{t\rightarrow +\infty}{\Phi_t}=0$.
\end{itemize}
\begin{itemize}
\item[$iii)$] There exist constants $c\not=0$ and $\tilde{c}\not=0$ such that
\begin{eqnarray*}
\lim\limits_{t\rightarrow +\infty}{e^{V^{\prime\prime}{(0)}t}\Phi_t}&=&c,
\end{eqnarray*}
\begin{eqnarray*}
\lim\limits_{t\rightarrow +\infty}{e^{V^{\prime\prime}{(0)}t}\psi_t}&=&\tilde{c},
\end{eqnarray*}
where $\Phi=\{\Phi_{t}\}_{t\geq 0}$ is the fundamental solution of the non autonomous system
\begin{eqnarray*}
d\Phi_{t}=-V^{\prime\prime}{(\psi_t)}\Phi_t dt
\end{eqnarray*}
for every $t\geq 0$ with initial condition $\Phi_{0}=1$.
\item[$iv)$]
\begin{eqnarray*}
\lim\limits_{t\rightarrow +\infty}{\Phi_t^2 \int\limits_{0}^{t}{\left(\frac{1}{\Phi_s}\right)^2d{{s}}}}=\frac{1}{2V^{\prime\prime}{(0)}}.
\end{eqnarray*}
\end{itemize}
\end{lemma}
For the proof of this lemma, see Lemma \ref{importante0}.
\begin{remark}
For the items $i)$ and $ii)$ in the Lemma \ref{asin}, we do not need the assumption that $V\in \mathcal{C}^3$; we need less regularity, $V\in \mathcal{C}^2$ is enough.
\end{remark}

The following lemma characterizes the distribution of the ``linear approximations".

\begin{lemma}\label{gauus}
Under the hypothesis of Theorem \ref{toy}, we have
\begin{eqnarray}\label{gaussfor}
\y=\Phi_{t}y_0+\sqrt{\epsilon} \Phi_{t}\int\limits_{0}^{t}{\frac{1}{\Phi(s)}d{W_{s}}}
\end{eqnarray}
for every $t\geq 0$,
where $\Phi=\{\Phi_{t}\}_{t\geq 0}$ is the fundamental solution of the non autonomous system
\begin{eqnarray*}
d\Phi_{t}=-V^{\prime\prime}{(\psi_t)}\Phi_t dt
\end{eqnarray*}
for every $t\geq 0$ with initial condition $\Phi_{0}=1$.
\end{lemma}

\begin{proof}
It follows from It\^o's formula. For details check \cite{KS}.
\end{proof}
Using the decomposition (\ref{gaussfor}) of the process $y^{\epsilon}$ into a deterministic part and a mean-zero martingale and using It\^o's isometry for Wiener's integral, we obtain
\begin{eqnarray*}
\mathbb{E}\left[\y\right]&=& \Phi_t y_0,\\
\mathbb{V}\left[\y\right]&=&\epsilon \Phi_t^2 \int\limits_{0}^{t}{\left(\frac{1}{\Phi_s}\right)^2d{{s}}}.
\end{eqnarray*}

By Lemma \ref{gauus}, we have that for each $\epsilon>0$ and $t>0$ fixed, $\y$ is a random variable with Normal distribution
with mean
\begin{eqnarray*}
\nu^{\epsilon}_{t}&:=&\Phi_t y_0
\end{eqnarray*}
and variance
\begin{eqnarray*}
\eta^{\epsilon}_{t}&:=&\epsilon \Phi_t^2 \int\limits_{0}^{t}{\left(\frac{1}{\Phi_s}\right)^2d{{s}}}.
\end{eqnarray*}

\begin{corollary} Let us assume the hypothesis of Theorem \ref{toy} and  let $\epsilon>0$ be fixed. Then the random variable $\y$ converges in distribution as $t \to \infty$  to a Gaussian random variable $\N^{\epsilon}$ with mean zero
and variance $\frac{\epsilon}{2V^{\prime\prime}{(0)}}$.
\end{corollary}

\begin{proof}
It follows from the item $ii)$ and item $iv)$ of Lemma \ref{asin}.
\end{proof}

Now, we have all the tools in order to prove Theorem \ref{toy}.

\begin{proof}[Proof of Theorem \ref{toy}]
For each $\epsilon>0$ and $t>0$, we define
\begin{eqnarray*}
d^{\epsilon}{(t)}&=&\norm{\N\left(\nu^{\epsilon}_{t},\eta^{\epsilon}_{t}\right)-\N{\left(0,\frac{\epsilon}{2V^{\prime\prime}(0)}\right)}}
\end{eqnarray*}
and
\begin{eqnarray*}
D^{\epsilon}{(t)}&=&
\norm{\N{\left(\sqrt{\frac{2V^{\prime\prime}{(0)}}{\epsilon}}y_0 \Phi_t,1\right)}-\N{\left(0,1\right)}}.
\end{eqnarray*}
Using  triangle's inequality and Lemma \ref{lemma1}, for each $\epsilon>0$ and $t>0$ we obtain
\begin{eqnarray*}
d^{\epsilon}{(t)}&\leq & D^{\epsilon}{(t)}+
\norm{\N{\left(0,{2V^{\prime\prime}(0)}{\Phi_t}^2 I_t\right)}-\N{\left(0,1\right)}}
\end{eqnarray*}
and
\begin{eqnarray*}
D^{\epsilon}{(t)}&\leq & d^{\epsilon}{(t)}+
\norm{\N{\left(0,{2V^{\prime\prime}(0)}{\Phi_t}^2 I_t\right)}-\N{\left(0,1\right)}},
\end{eqnarray*}
where $I_t=\int\limits_{0}^{t}{\left(\frac{1}{\Phi_s}\right)^2d{{s}}}$. Therefore,
\begin{eqnarray*}
|d^{\epsilon}{(t)}-D^{\epsilon}{(t)}| &\leq &
\norm{\N{\left(0,{2V^{\prime\prime}(0)}{\Phi_t}^2 I_t\right)}-\N{\left(0,1\right)}}.
\end{eqnarray*}
For each $\epsilon>0$ let us define
\begin{eqnarray*}
t_{\epsilon}:=\frac{1}{2V^{\prime\prime}(0)}\left(\ln\left(\frac{1}{\epsilon}\right)+b_0\right)
\end{eqnarray*} and
\begin{eqnarray*}
w_{\epsilon}:=\frac{1}{V^{\prime\prime}(0)}
\end{eqnarray*}
with
$b_0:=\ln\left(2V^{\prime\prime}(0)y^2_0\right)$.
For every $b\in \R$, we define $t_{\epsilon}(b)=t_{\epsilon}+bw_{\epsilon}$. We  take $\epsilon_b>0$ such that $t_{\epsilon}(b)>0$ for every $0<\epsilon<\epsilon_b$.
 Using Lemma \ref{adicionalvar}, we obtain
\begin{eqnarray*}
\lim\limits_{\epsilon\rightarrow 0}|d^{\epsilon}{(t_{\epsilon}(b))}-D^{\epsilon}{(t_{\epsilon}(b))}|&=&0
\end{eqnarray*}
for every $b\in \R$. Let us consider the function $G:\R\rightarrow [0,1]$ defined by
\begin{eqnarray*}
G{(b)}:=\norm{\N{(ce^{-b},1)}-\N{(0,1)}},
\end{eqnarray*}
where $c\not=0$ is the constant of item $iii)$ in Lemma \ref{asin}.
Observe that
\begin{eqnarray*}
D^{\epsilon}{(t_{\epsilon}(b))}=
\norm{\N{\left(e^{V^{\prime\prime}(0) t_{\epsilon}(b)}\Phi_{t_{\epsilon}(b)}e^{-b},1\right)}-\N{(0,1)}}
\end{eqnarray*}
for every $b\in \R$ and $0<\epsilon<\epsilon_b$. Therefore, by item $iii)$ of Lemma \ref{asin} and by Lemma \ref{adicional},
we have
\begin{eqnarray*}
\lim\limits_{\epsilon\rightarrow 0}D^{\epsilon}{(t_{\epsilon}(b))}=G(b)
\end{eqnarray*}
for every $b\in \R$. By Lemma \ref{lemma2}, we have
$\lim\limits_{b\rightarrow +\infty}{G(b)}=0$ and $\lim\limits_{b\rightarrow -\infty}{G(b)}=1$.
Consequently, the theorem is proved.
\end{proof}

\begin{corollary}[The First Order Approximations]\label{first}
Let us consider the process $y=\{ {y_t} \}_{t \geq 0}$
which is given by the solution of the following linear stochastic differential equation,
\begin{eqnarray*}
d{y_t}&=&-V^{\prime\prime}{(\psi_t)}{y^{}_t}dt+dW_t,\\
{y_0}&=&0
\end{eqnarray*}
for $t \geq 0$, where
$\{W_t\}_{t \geq 0}$ is a standard Brownian motion, $V$ is a regular potential and $\{\psi_t\}_{t\geq 0}$ is given by the solution of the deterministic differential equation (\ref{dde1}).
For every $\epsilon>0$ fixed, let us define $z^{\epsilon}_t=\psi_t+\sqrt{\epsilon}y_t$ for every $t\geq 0$.
Then the family $\{z^{\epsilon}\}_{\epsilon>0}$
presents profile cut-off in the sense of Definition \ref{profile} with respect to the total variation distance when $\epsilon$ goes to zero.
The profile function $G:\mathbb{R}\rightarrow \mathbb{R}$ is given by
\begin{eqnarray*}
G(b)&:=&\norm{\N{\left(\tilde{c} e^{-b},1\right)}-\N{\left(0,1\right)}},
\end{eqnarray*}
where $\tilde{c}$ is the nonzero constant given by
\begin{eqnarray*}
\lim\limits_{t\rightarrow +\infty}{e^{V^{\prime\prime}{(0)}t}\psi_t}&=:&\tilde{c}.
\end{eqnarray*}
and the cut-off time $t_{\epsilon}$ and window time $w_{\epsilon}$ are given by
\begin{eqnarray*}
t_{\epsilon}&:=&\frac{1}{2V^{\prime\prime}(0)}\left[\ln\left(\frac{1}{\epsilon}\right)+\ln\left(2V^{\prime\prime}(0)\right) \right],\\
w_{\epsilon}&:=&\frac{1}{V^{\prime\prime}(0)}.
\end{eqnarray*}
\end{corollary}
The proof of Theorem  \ref{toy} can be adapted in order to prove this corollary in a straight-forward way, so we omit it.
In what follows, we call the processes $\{z^\epsilon\}_{\epsilon>0}$  the ``linear approximations''.
\begin{remark}\label{short}
The constants $c$ and $\tilde{c}$ obtained in the item $iii)$ of Lemma \ref{asin} depend on the initial condition of the semiflow $\psi=\{\psi_t\}_{t\geq 0}$.
Theorem \ref{toy} and Corollary \ref{first} remain true without altering the cut-off time and the profile function if we take as window time $w^{\prime}_{\epsilon}=w_{\epsilon}+\delta_\epsilon$ for each $\epsilon>0$, where $\{\delta_{\epsilon}\}_{\epsilon>0}$ is any sequence of real numbers such that
$\lim\limits_{\epsilon\rightarrow 0}{\delta_{\epsilon}}=0$.
\end{remark}

\section{The Gradient Case}\label{gradcase}
From now on and up to the end of this section we will use the following notations and names.
\begin{definition} ~
\begin{itemize}
\item[$a)$] The stochastic process $x^{\epsilon}:=\left\{x^{\epsilon}_t\right\}_{t\geq 0}$ defined in Theorem \ref{main} is called the It\^o diffusion.
\item[$a)$] The semiflow $\psi:=\left\{\psi_t  \right\}_{t\geq 0}$ defined by the differential equation (\ref{dde1}) is called the zeroth order approximation of $x^{\epsilon}$.
\item[$c)$] The stochastic Markov process $z^{\epsilon}:=\left\{z^{\epsilon}_t\right\}_{t\geq 0}$ defined in Corollary \ref{first} is called the first order approximation of $x^{\epsilon}$.
\end{itemize}
\end{definition}

The following lemma will give us the existence of a stationary probability measure for the It\^o diffusion
$x^{\epsilon}=\left\{x^{\epsilon}_t\right\}_{t\geq 0}$.

\begin{lemma}\label{density}
Let $V$ be a regular potential and for every $\epsilon>0$, let us consider the It\^o diffusion $x^{\epsilon}=\{\x\}_{t\geq 0}$ which is given by the following stochastic differential equation,
\begin{eqnarray*}
d{x^{\epsilon}_t}&=&-V^{\prime}({x^{\epsilon}_t})dt+\sqrt{\epsilon}dW_t,\\
{x^{\epsilon}_0}&=&x_0
\end{eqnarray*}
for $t \geq 0$, where $x_0$ is a deterministic initial condition in $\mathbb{R}\setminus \{0\}$ and
$\{W_t\}_{t \geq 0}$ is a standard Brownian motion. Let us assume that
\begin{eqnarray*}
\lim\limits_{|x|\rightarrow +\infty}{\left |V^{\prime}(x)\right |}&=&+\infty.
\end{eqnarray*}
Then, for every $\epsilon>0$ fixed, when $t\rightarrow +\infty$  the probability distribution of $x^{\epsilon}_t$, $\mathbb{P}(x^{\epsilon}_t\in \cdot)$ converges in distribution to the stationary probability measure $\mu^{\epsilon}$ given by
\begin{eqnarray*}
\mu^{\epsilon}(dx)&=&\frac{e^{-\frac{2}{\epsilon}V(x)}dx}{M^{\epsilon}},
\end{eqnarray*}
where $M^{\epsilon}=\int\limits_{\mathbb{R}}{e^{-\frac{2}{\epsilon}V(z)}dz}$.
\end{lemma}
For the proof of this lemma and further considerations, see \cite{JA} and \cite{SI}.
Now we will restrict our potential to the class of coercive regular potentials.
\begin{definition}[Coercive Regular Potential]
Let $V$ be a regular potential. We say that $V$ is a coercive regular potential if there exists $\delta>0$ such that  $V^{\prime\prime}(x)\geq \delta$ for every $x\in \mathbb{R}$.
\end{definition}

In the class of coercive regular potentials, we restrict ourselves to the class of potentials with bounded second and third derivatives which we call smooth coercive regular potentials.
\begin{definition}[Smooth Coercive Regular Potential]
Let $V$ be a coercive regular potential. We say that $V$ is a smooth coercive regular potential if
\begin{eqnarray*}
\kappa_2:=\|V^{\prime\prime}\|_{\infty}:=\sup\limits_{x\in\R}{|V^{\prime\prime}(x)|}<+\infty,
\end{eqnarray*} and
\begin{eqnarray*}
\kappa_3:=\|V^{\prime\prime\prime}\|_{\infty}:=\sup\limits_{x\in\R}{|V^{\prime\prime\prime}(x)|}<+\infty.
\end{eqnarray*}
\end{definition}
The following lemma tells us that the stationary probability measure of the It\^o diffusion $\{x^{\epsilon}_t\}_{t\geq 0}$ is well approximated in total variation distance by the Normal distribution with mean zero and variance $\frac{\epsilon}{2V^{\prime\prime}(0)}$.
\begin{lemma} \label{invariantes}
Let $V$ be a coercive regular potential,
then
\begin{eqnarray*}
 \lim\limits_{\epsilon\rightarrow 0}\norm{\mu^{\epsilon}-\N^{\epsilon}}=0,
\end{eqnarray*}
where $\N^{\epsilon}$ is a Normal distribution with mean zero and variance $\frac{\epsilon}{2V^{\prime\prime}(0)}$.
\end{lemma}
\begin{proof}
Let $0<\eta<V^{\prime\prime}(0)$ be fixed. By Lemma \ref{density}, the $\mu^{\epsilon}(dx)=\frac{e^{-\frac{2}{\epsilon}V(x)}dx}{M^{\epsilon}}$ is a well defined probability measure
on $\left(\mathbb{R},\mathcal{B}\left(\mathbb{R}\right)\right)$.
Then
\begin{eqnarray*}
\norm{\mu^{\epsilon}-\N^{\epsilon}}&=&\frac{1}{2}\int\limits_{\mathbb{R}}
{\left|\frac{e^{-\frac{2}{\epsilon}V(x)}}{M^{\epsilon}}-\frac{e^{-\frac{2}{\epsilon}\frac{V^{\prime\prime}(0)x^2}{2}}}{N^{\epsilon}}\right|dx},
\end{eqnarray*}
where
$M^{\epsilon}=\int\limits_{\mathbb{R}}{e^{-\frac{2}{\epsilon}V(x)}dx}$ and
$N^{\epsilon}=\int\limits_{\mathbb{R}}{e^{-\frac{2}{\epsilon}\frac{V^{\prime\prime}(0)x^2}{2}}dx}=\sqrt{\frac{\pi\epsilon}{V^{\prime\prime}(0)}}$.

By triangle's inequality, we have
\begin{eqnarray*}
\norm{\mu^{\epsilon}-\N^{\epsilon}}&\leq &\frac{1}{2}\int\limits_{\mathbb{R}}
{\left|\frac{e^{-\frac{2}{\epsilon}V(x)}}{M^{\epsilon}}-\frac{e^{-\frac{2}{\epsilon}V(x)}}{N^{\epsilon}}\right|dx}+
\frac{1}{2}\int\limits_{\mathbb{R}}
{\left|\frac{e^{-\frac{2}{\epsilon}V(x)}}{N^{\epsilon}}-\frac{e^{-\frac{2}{\epsilon}\frac{V^{\prime\prime}(0)x^2}{2}}}{N^{\epsilon}}\right|dx}\\
&=&\frac{\left|M^{\epsilon}-N^{\epsilon} \right|}{2N^{\epsilon}}+
\frac{1}{2N^{\epsilon}}\int\limits_{\mathbb{R}}
{\left|{e^{-\frac{2}{\epsilon}V(x)}}{}-{e^{-\frac{2}{\epsilon}\frac{V^{\prime\prime}(0)x^2}{2}}}{}\right|dx}\\
&\leq &
\frac{1}{N^{\epsilon}}\int\limits_{\mathbb{R}}
{\left|{e^{-\frac{2}{\epsilon}V(x)}}{}-{e^{-\frac{2}{\epsilon}\frac{V^{\prime\prime}(0)x^2}{2}}}{}\right|dx}.
\end{eqnarray*}

Recall that $V$ is a coercive regular potential, so  there exists $\delta>0$ such that $V^{\prime\prime}(x)\geq \delta>0$ for every $x\in \R$. Then, it follows that
\begin{eqnarray*}
\lim\limits_{\epsilon\rightarrow 0}{\frac{1}{N^{\epsilon}}\int\limits_{\{x\in \mathbb{R}:|x|\geq \beta\}}
{\left|{e^{-\frac{2}{\epsilon}V(x)}}{}-{e^{-\frac{2}{\epsilon}\frac{V^{\prime\prime}(0)x^2}{2}}}{}\right|dx}}&=&0
\end{eqnarray*}
for every $\beta>0$.
By the continuity of $V^{\prime \prime }$ at zero, there exists $\delta_{\eta}>0$ such that
\begin{eqnarray*}
|V^{\prime\prime}(x)-V^{\prime\prime}(0)|<\eta
\end{eqnarray*} for every $|x|<\delta_{\eta}$.

Also, by Taylor's Theorem, we have
that
$V^{\prime\prime}(x)=\frac{V^{\prime\prime}(\xi_x)x^2}{2}$ for every $|x|<\delta_{\eta}$ where $|\xi_x|< |x|$.
Then,
\[
\begin{split}
\frac{1}{N^{\epsilon}}\int\limits_{-\delta_{\eta}}^{\delta_{\eta}}
\Big|{e^{-\frac{2}{\epsilon}V(x)}}{}-&{e^{-\frac{2}{\epsilon}\frac{V^{\prime\prime}(0)x^2}{2}}}{}\Big|dx
=  \frac{1}{N^{\epsilon}}\int\limits_{-\delta_{\eta}}^{\delta_{\eta}}
\left|
e^{-\frac{2}{\epsilon}\frac{V^{\prime\prime}(\xi_x)x^2}{2}}-{e^{-\frac{2}{\epsilon}\frac{V^{\prime\prime}(0)x^2}{2}}}
\right|dx\\
&\leq  {\frac{1}{\epsilon N^{\epsilon}}\int\limits_{-\delta_{\eta}}^{\delta_{\eta}}
{
{x^2e^{-\frac{2}{\epsilon}\frac{\delta x^2}{2}}}
}\left| V^{\prime\prime}(\xi_x)-V^{\prime\prime}(0) \right|dx}\\
&\leq  {\frac{\eta}{\epsilon N^{\epsilon}}\int\limits_{-\delta_{\eta}}^{\delta_{\eta}}
{
{x^2e^{-\frac{2}{\epsilon}\frac{\delta x^2}{2}}}
}dx} \leq  {\frac{\eta\sqrt{V^{\prime\prime}(0)}}{\sqrt{\pi}(2\delta)^{\frac{3}{2}}}\int\limits_{-\delta_{\eta}\sqrt{\frac{2\delta}{\epsilon}}}^{\delta_{\eta}\sqrt{\frac{2\delta}{\epsilon}}}
{
{x^2e^{-\frac{x^2}{2}}}
}dx}\\
&\leq  {\frac{\eta\sqrt{V^{\prime\prime}(0)}}{\sqrt{\pi}(2\delta)^{\frac{3}{2}}}\int\limits_{\R}
{
{x^2e^{-\frac{x^2}{2}}}
}dx}.
\end{split}
\]
Consequently, first taking $\epsilon\rightarrow 0$ and then $\eta\rightarrow 0$ we obtain the result.
\end{proof}

The following proposition will give us a quantitative estimation of the distance of the paths between  the It\^o diffusion and the zeroth order and first order approximations.

\begin{proposition}[Zero Order \& First Order Approximation]\label{order}
Let us assume that $V$ is a smooth coercive regular potential. Let us denote
$B_t=\sup\limits_{0\leq s\leq t}{|W_s|}$ for every $t\geq 0$.
\begin{itemize}
\item[$i)$] For every $\epsilon>0$ and $t\geq 0$, we have $\left|\x-\psi_t\right|\leq \sqrt{\epsilon}B_t(\kappa_{2}t+1)$. We call this estimate the zeroth order estimate.
\item[$ii)$] For every $\epsilon>0$ and $t\geq 0$, it follows that $\left|\x-\psi_t-\sqrt{\epsilon}y_t\right|\leq {\epsilon}B^2_t \kappa_3(\kappa_{2}t+1)^2t$.
We call this estimate the first order estimate.
\end{itemize}
\end{proposition}

\begin{proof}
First we prove item $i)$.
Let $\epsilon>0$ and $t\geq 0$ be fixed. It follows that
\begin{eqnarray*}
\x-\psi_t &=&-\int\limits_{0}^{t}{\left(V^{\prime}(x^{\epsilon}_s)-V^{\prime}(\psi_s)\right)ds}+\sqrt{\epsilon}W_t\\
&=& -\int\limits_{0}^{t}{V^{\prime\prime}(\theta^{\epsilon}_s)\left(x^{\epsilon}_s-\psi_s\right)ds}+\sqrt{\epsilon}W_t\\
&=& -\sqrt{\epsilon}\int\limits_{0}^{t}{V^{\prime\prime}(\theta^{\epsilon}_s)W_s
e^{-\int\limits_{s}^{t}{V^{\prime\prime}(\theta^{\epsilon}_r)dr}}ds}+\sqrt{\epsilon}W_t,
\end{eqnarray*}
where the second equality follows from the Mean Value Theorem,
$\theta_s^\epsilon$ is between the minimum of $\psi_s$ and $x_s^\epsilon$ and the maximum of $\psi_s$ and $x_s^\epsilon$, and the third equality follows from the variation of parameters method. Therefore, using Gronwall's inequality we obtain,
$\left|\x-\psi_t\right|\leq \sqrt{\epsilon}B_t(\kappa_{2}t+1)$.

Now we prove item $ii)$.
Again, let $\epsilon>0$ and $t\geq 0$ be fixed. It follows that
\begin{eqnarray*}
\x-\psi_t -\sqrt{\epsilon}y_t&=& -\int\limits_{0}^{t}{\left[V^{\prime}(x^{\epsilon}_s)-V^{\prime}(\psi_s)-V^{\prime\prime}(\psi_s)\sqrt{\epsilon}y_s\right]ds}\\
&=& -\int\limits_{0}^{t}{\left[V^{\prime\prime}(\theta^{\epsilon}_s)\left(x^{\epsilon}_s-\psi_s\right)-V^{\prime\prime}(\psi_s)\sqrt{\epsilon}y_s\right]ds}\\
&=& -\int\limits_{0}^{t}{V^{\prime\prime}(\psi_s)(x^{\epsilon}_s-\psi_s-\sqrt{\epsilon}y_t)ds}-\\
&&\int\limits_{0}^{t}{\left(V^{\prime\prime}(\theta^{\epsilon}_s)-V^{\prime\prime}(\psi_s)\right)(x^{\epsilon}_s-\psi_s)ds},\\
\end{eqnarray*}
where the second equality comes from the Mean Value Theorem, $\theta_s^\epsilon$ is between the minimum of $\psi_s$ and $x_s^\epsilon$ and the maximum of $\psi_s$ and $x_s^\epsilon$. Let us define $$e_t:=\int\limits_{0}^{t}{\left(V^{\prime\prime}(\theta^{\epsilon}_s)-V^{\prime\prime}(\psi_s)\right)(x^{\epsilon}_s-\psi_s)ds}.$$ Again, using the Mean Value Theorem and the zeroth order estimate already proved, we have
\begin{eqnarray*}
|e_t|\leq \int\limits_{0}^{t}{\kappa_3(x^{\epsilon}_s-\psi_s)^2ds}\leq \epsilon B^2_t\kappa_3(\kappa_{2}t+1)^2t
\end{eqnarray*}
for every $t\geq 0$.
Consequently, using the
variation of parameters method and Gronwall's inequality we obtain
\begin{eqnarray*}
 \left|\x-\psi_t-\sqrt{\epsilon}y_t\right| &\leq & \epsilon B^2_t \kappa_3(\kappa_{2}t+1)^3t.
 \end{eqnarray*}

\end{proof}
The next proposition will allows us to prove that the total variation distance of two first order approximations with (random or deterministic) initial conditions that are close enough is negligible. In order to do that, we will need to keep track of the initial condition of the solution of various equations. Let $X$ be a random variable in $\mathbb R$ and let $T>0$. Let $\{\psi_{t}(X)\}_{t \geq 0}$ denote the solution of
\begin{eqnarray*}
d{\psi_t(X)}&=&- V^\prime(\psi_t(X))dt,\\
\psi_0(X)&=&X.
\end{eqnarray*}
Let $\{y_t(X,T)\}_{t \geq 0}$ be the solution of the stochastic differential equation
\begin{eqnarray*}
d{y_t(X,T)}&=&-V^{\prime\prime}(\psi_t(X))y_t(X,T)dt+dW_{t+T},\\
y_0(X,T)&=&0
\end{eqnarray*}
and define $\{z^\epsilon_t(X,T)\}_{t\geq0}$ as $z^\epsilon_t(X,T) := \psi_t(X) + \sqrt \epsilon y_t(X,T)$. In what follows, we will always take $T = t_\epsilon(b) := t_\epsilon + b w_\epsilon>0$ for every $\epsilon>0$ small enough, so we will omit it from the notation.

\begin{proposition}[Linear Coupling]\label{lincou}
Let us assume that $V$ is a smooth coercive regular potential. Take $\delta_{\epsilon}:=\epsilon^{\gamma}$, where $0<\gamma<1$.
Let us denote by $z^{\epsilon}(X):=\{z^{\epsilon}_t(X)\}_{t\geq 0}$ the first order approximation with initial random condition $X$.
Then, for every $b\in \R$ it follows that
\begin{eqnarray*}
\lim\limits_{\epsilon\rightarrow 0}{\norm{
z^{\epsilon}_{b\delta_{\epsilon}}\left(x^{\epsilon}_{t_{\epsilon}(b)}\right)-z^{\epsilon}_{b\delta_{\epsilon}}\left(z^{\epsilon}_{t_{\epsilon}(b)}\right)}}&=&0,
\end{eqnarray*}
where for each $\epsilon>0$, $t_{\epsilon}$ and $w_{\epsilon}$ are defined in Corollary \ref{first} and for each $b\in \R$, $\epsilon_b>0$ is small enough so that $t_{\epsilon}(b):=t_{\epsilon}+bw_{\epsilon}> 0$ for every $0<\epsilon<\epsilon_b$.
\end{proposition}

\begin{proof}
By It\^o's formula we obtain
\begin{eqnarray*}
z^{\epsilon}_{b\delta_{\epsilon}}\left(x^{\epsilon}_{t_{\epsilon}(b)}\right)&=&
\Phi_{b\delta_{\epsilon}}x^{\epsilon}_{t_{\epsilon}(b)}+\sqrt{\epsilon} \Phi_{b\delta_{\epsilon}}\int\limits_{0}^{b\delta_{\epsilon}}{\frac{1}{\Phi(s)}d{\left(W_{t_{\epsilon}(b)+s}-W_{t_{\epsilon}(b)}\right)}},
\end{eqnarray*}

\begin{eqnarray*}
z^{\epsilon}_{b\delta_{\epsilon}}\left(z^{\epsilon}_{t_{\epsilon}(b)}\right)&=&
\Phi_{b\delta_{\epsilon}}z^{\epsilon}_{t_{\epsilon}(b)}+\sqrt{\epsilon} \Phi_{b\delta_{\epsilon}}\int\limits_{0}^{b\delta_{\epsilon}}{\frac{1}{\Phi(s)}d{\left(W_{t_{\epsilon}(b)+s}-W_{t_{\epsilon}(b)}\right)}}
\end{eqnarray*}
for every $0<\epsilon<\epsilon_b$,
where $\Phi=\{\Phi_{t}\}_{t\geq 0}$ is the fundamental solution of the non-autonomous system
\begin{eqnarray*}
d\Phi_{t}=-V^{\prime\prime}{(\psi_t+{t}_\epsilon(b))}\Phi_t dt
\end{eqnarray*}
for every $t\geq 0$ with initial condition $\Phi_{0}=1$.
Applying Lemma \ref{conditional} with $X:=\Phi_{b\delta_{\epsilon}}x^{\epsilon}_{t_{\epsilon}(b)}$, $Y:=\Phi_{b\delta_{\epsilon}}z^{\epsilon}_{t_{\epsilon}(b)}$ and $Z:=\sqrt{\epsilon} \Phi_{b\delta_{\epsilon}}\int\limits_{0}^{b\delta_{\epsilon}}{\frac{1}{\Phi(s)}d{\left(W_{t_{\epsilon}(b)+s}-W_{t_{\epsilon}(b)}\right)}}$,
$\mathcal{G}=\sigma\left(X,Y\right)$ and $(\Omega,\mathcal{F},\mathbb{P})$ the canonical probability space of the Brownian motion, we obtain
\begin{eqnarray*}
\norm{
z^{\epsilon}_{b\delta_{\epsilon}}\left(x^{\epsilon}_{t_{\epsilon}(b)}\right)-z^{\epsilon}_{b\delta_{\epsilon}}\left(z^{\epsilon}_{t_{\epsilon}(b)}\right)}&\leq &
\frac{1}{\sqrt{{2\pi}{\epsilon}\int\limits_{0}^{b\delta_{\epsilon}}{\left(\frac{1}{\Phi(s)}\right)^2d{{s}}}}}\mathbb{E}{\left[\left|x^{\epsilon}_{t_{\epsilon}(b)}-z^{\epsilon}_{t_{\epsilon}(b)}\right|\right]}.
\end{eqnarray*}
Using Proposition \ref{order}, we obtain
\begin{eqnarray*}
\norm{
z^{\epsilon}_{b\delta_{\epsilon}}\left(x^{\epsilon}_{t_{\epsilon}(b)}\right)-z^{\epsilon}_{b\delta_{\epsilon}}\left(z^{\epsilon}_{t_{\epsilon}(b)}\right)}&\leq &
\sqrt{\frac{{\epsilon}}{{{2\pi}{}\int\limits_{0}^{b\delta_{\epsilon}}{\left(\frac{1}{\Phi(s)}\right)^2d{{s}}}}}} \times \\
&&\kappa_3\left(\kappa_{2}t_{\epsilon}(b)+1\right)^3t_{\epsilon}(b)
\mathbb{E}{\left[B^2_{t_{\epsilon}(b)}\right]}.
\end{eqnarray*}
Using the fact that for each $\epsilon>0$, $\delta_{\epsilon}=\epsilon^{\gamma}$ for some $0<\gamma<1$, $\Phi_0=1$, the Intermediate Value Theorem for integrals and Lemma \ref{hopital} we obtain the result.
\end{proof}

The following proposition will permit us to change the probability measure in a small interval of time in order to compare the total variation distance of the It\^o diffusion and the first order approximation with a random initial condition.

\begin{proposition}[Short Time Change of Measure]\label{shorttime}
Let us assume the same hypothesis of Proposition \ref{lincou}. Then for each $b\in \mathbb{R}$
\begin{eqnarray*}
\lim\limits_{\epsilon\rightarrow 0}{\norm{
x^{\epsilon}_{b\delta_{\epsilon}}\left(x^{\epsilon}_{t_{\epsilon}(b)}\right)-z^{\epsilon}_{b\delta_{\epsilon}}\left(x^{\epsilon}_{t_{\epsilon}(b)}\right)}}&=&0,
\end{eqnarray*}
where $\delta_\epsilon=\epsilon^{\gamma}$ for some $\gamma>0$.
\end{proposition}

\begin{proof}
We will use Cameron-Martin-Girsanov Theorem and Novikov's Theorem. For the precise statements of these theorems we use here, see \cite{FI1} and \cite{KU}.
Let $\epsilon>0$, $t\geq 0$  and $b\in \R$ be fixed. Let us define $\gamma^{\epsilon}_{t}:=\frac{V^{\prime}\left(x^{\epsilon}_t\right)}{\sqrt{\epsilon}}$ and
$\Gamma^{\epsilon}_{t}:=\frac{\left(V^{\prime}\left(\psi_t\right)-V^{\prime\prime}(\psi_t)\psi_t+V^{\prime\prime}(\psi_t)z^{\epsilon}_t\right)}{\sqrt{\epsilon}}$.
Then, for every $\epsilon>0$ and $t>0$ it follows that
\begin{eqnarray*}
\left(\gamma^{\epsilon}_{t}\right)^2
&\leq &2\kappa^2_2\frac{\left(x^{\epsilon}_t-\psi_t\right)^2}{{\epsilon}}+
2\kappa^2_2\frac{\left(\psi_t\right)^2}{{\epsilon}}\\
&\leq & 4\kappa^2_2 B^2_t\left(\kappa_2 t^2+1\right)+
2\kappa^2_2\frac{\left(\psi_t\right)^2}{{\epsilon}}\\
\end{eqnarray*}
and
\begin{eqnarray*}
\left(\Gamma^{\epsilon}_{t}\right)^2
&\leq &2\kappa^2_2\left(y_t\right)^2+
2\kappa^2_2\frac{\left(\psi_t\right)^2}{{\epsilon}}\\
&\leq & 4\kappa^2_2 B^2_t\left(\kappa_2 t^2+1\right)+
2\kappa^2_2\frac{\left(\psi_t\right)^2}{{\epsilon}}.\\
\end{eqnarray*}
Let us define $I^{\epsilon}(b):=\left[t_{\epsilon}(b),t_{\epsilon}(b)+b\delta_{\epsilon}\right]$. Then, for every $\epsilon>0$ it follows that
\begin{eqnarray*}
\int\limits_{I(\epsilon)}\left(\gamma^{\epsilon}_{t}\right)^2dt
&\leq & 4b\kappa^2_2 \delta_{\epsilon}\left(\kappa_2 \left(t_{\epsilon}(b)+b\delta_{\epsilon}\right)^2+1\right)\sup\limits_{t\in I^{\epsilon}(b)}{B^2_t}+
2b\kappa^2_2 \delta_{\epsilon}\frac{\sup\limits_{t\in I^{\epsilon}(b)}\left(\psi_t\right)^2}{{\epsilon}}.\\
\end{eqnarray*}
and
\begin{eqnarray*}
\int\limits_{I(\epsilon)}\left(\Gamma^{\epsilon}_{t}\right)^2dt
&\leq & 4b\kappa^2_2 \delta_{\epsilon}\left(\kappa_2 \left(t_{\epsilon}(b)+b\delta_{\epsilon}\right)^2+1\right)\sup\limits_{t\in I^{\epsilon}(b)}{B^2_t}+
2b\kappa^2_2 \delta_{\epsilon}\frac{\sup\limits_{t\in I^{\epsilon}(b)}\left(\psi_t\right)^2}{{\epsilon}}.\\
\end{eqnarray*}
Using Lemma \ref{porque}, there exists a constant $c>0$ such that
\begin{eqnarray*}
\int\limits_{I(\epsilon)}\left(\gamma^{\epsilon}_{t}\right)^2dt
&\leq & 4b\kappa^2_2 \delta_{\epsilon}\left(\kappa_2 \left(t_{\epsilon}(b)+b\delta_{\epsilon}\right)^2+1\right)\sup\limits_{t\in I^{\epsilon}(b)}{B^2_t}+
2bc\kappa^2_2 \delta_{\epsilon}
\end{eqnarray*}
and
\begin{eqnarray*}
\int\limits_{I(\epsilon)}\left(\Gamma^{\epsilon}_{t}\right)^2dt
&\leq & 4b\kappa^2_2 \delta_{\epsilon}\left(\kappa_2 \left(t_{\epsilon}(b)+b\delta_{\epsilon}\right)^2+1\right)\sup\limits_{t\in I^{\epsilon}(b)}{B^2_t}+
2bc\kappa^2_2 \delta_{\epsilon}
\end{eqnarray*}
for $\epsilon>0$ small enough. Consequently, for any constant $\rho>0$ it follows that
\begin{eqnarray*}
\mathbb{E}\left\{\exp\left[\rho\int\limits_{t_{\epsilon}(b)}^{t_{\epsilon}(b)+b\delta_{\epsilon}}{\left(\gamma^{\epsilon}_s\right)^2 ds}\right]\right\}<+\infty
\end{eqnarray*}
and
\begin{eqnarray*}
\mathbb{E}\left\{\exp\left[\rho\int\limits_{t_{\epsilon}(b)}^{t_{\epsilon}(b)+b\delta_{\epsilon}}{\left(\Gamma^{\epsilon}_s\right)^2 ds}\right]\right\}<+\infty
\end{eqnarray*}
for $\epsilon>0$ small enough.
From Novikov's Theorem it follows that

\begin{eqnarray*}
\frac{d\mathbb{P}^1_{t_{\epsilon}(b)+b\delta_{\epsilon}}}{d\mathbb{P}_{t_{\epsilon}(b)+b\delta_{\epsilon}}}&:=&
\exp\left\{\int\limits_{t_{\epsilon}(b)}^{t_{\epsilon}(b)+b\delta_{\epsilon}}
{\gamma^{\epsilon}_s dW_s}-\frac{1}{2}
\int\limits_{t_{\epsilon}(b)}^{t_{\epsilon}(b)+b\delta_{\epsilon}}
{\left(\gamma^{\epsilon}_s\right)^2 ds}
\right\},\\
\frac{d\mathbb{P}^2_{t_{\epsilon}(b)+b\delta_{\epsilon}}}{d\mathbb{P}_{t_{\epsilon}(b)+b\delta_{\epsilon}}}&:=&
\exp\left\{\int\limits_{t_{\epsilon}(b)}^{t_{\epsilon}(b)+b\delta_{\epsilon}}
{\Gamma^{\epsilon}_s dW_s}-\frac{1}{2}
\int\limits_{t_{\epsilon}(b)}^{t_{\epsilon}(b)+b\delta_{\epsilon}}
{\left(\Gamma^{\epsilon}_s\right)^2 ds}
\right\},
\end{eqnarray*}
are well defined and they define true probability measures
$\mathbb{P}^{i}_{t_{\epsilon}(b)+b\delta_{\epsilon}}$, $i\in\{1,2\}$.
From now on and up to the end of this proof we will use the notation
$\mathbb{P}^{i}:=\mathbb{P}^{i}_{t_{\epsilon}(b)+b\delta_{\epsilon}}$, $i\in\{1,2\}$ and
$\mathbb{P}:=\mathbb{P}_{t_{\epsilon}(b)+b\delta_{\epsilon}}$.
Under the probability measure $\mathbb{P}^1$,
$W^1_t:=W_t-\int\limits_{t_{\epsilon}(b)}^{t}{\gamma^{\epsilon}_s ds}$ is a Brownian motion on the time interval $t_{\epsilon}(b) \leq t\leq t_{\epsilon}(b)+b\delta_{\epsilon}$.
Also, under the probability measure $\mathbb{P}^2$,
$W^2_t:=W_t-\int\limits_{t_{\epsilon}(b)}^{t}{\Gamma^{\epsilon}_s ds}$ is a Brownian motion on the time interval $t_{\epsilon}(b) \leq t\leq t_{\epsilon}(b)+b\delta_{\epsilon}$.
Consequently,
\begin{eqnarray*}
\frac{d\mathbb{P}^1}{d\mathbb{P}^2}&=&
\frac{\exp\left\{\int\limits_{t_{\epsilon}(b)}^{t_{\epsilon}(b)+b\delta_{\epsilon}}
{\gamma^{\epsilon}_s dW_s}-\frac{1}{2}
\int\limits_{t_{\epsilon}(b)}^{t_{\epsilon}(b)+b\delta_{\epsilon}}
{\left(\gamma^{\epsilon}_s\right)^2 ds}
\right\}}
{\exp\left\{\int\limits_{t_{\epsilon}(b)}^{t_{\epsilon}(b)+b\delta_{\epsilon}}
{\Gamma^{\epsilon}_s dW_s}-\frac{1}{2}
\int\limits_{t_{\epsilon}(b)}^{t_{\epsilon}(b)+b\delta_{\epsilon}}
{\left(\Gamma^{\epsilon}_s\right)^2 ds}
\right\}}\\
&=&\exp\left\{\int\limits_{t_{\epsilon}(b)}^{t_{\epsilon}(b)+b\delta_{\epsilon}}
{\left(\gamma^{\epsilon}_s-\Gamma^{\epsilon}_s\right) dW_s}-\frac{1}{2}
\int\limits_{t_{\epsilon}(b)}^{t_{\epsilon}(b)+b\delta_{\epsilon}}
{\left(\left(\gamma^{\epsilon}_s\right)^2-\left(\Gamma^{\epsilon}_s\right)^2\right) ds}
\right\}\\
&=&\exp\left\{\int\limits_{t_{\epsilon}(b)}^{t_{\epsilon}(b)+b\delta_{\epsilon}}
{\left(\gamma^{\epsilon}_s-\Gamma^{\epsilon}_s\right) dW^1_s}+\frac{1}{2}
\int\limits_{t_{\epsilon}(b)}^{t_{\epsilon}(b)+b\delta_{\epsilon}}
{\left(\Gamma^{\epsilon}_s-\gamma^{\epsilon}_s\right)^2 ds}
\right\}.
\end{eqnarray*}
By Pinsker's inequality and the mean-zero martingale property of the stochastic integral, we have for every
$t_{\epsilon}(b)\leq t\leq t_{\epsilon}(b)+b\delta_{\epsilon}$
\begin{eqnarray*}
{\norm{
x^{\epsilon}_{b\delta_{\epsilon}}\left(x^{\epsilon}_{t_{\epsilon}(b)}\right)-z^{\epsilon}_{b\delta_{\epsilon}}\left(x^{\epsilon}_{t_{\epsilon}(b)}\right)}}
&\leq &
\mathbb{E}_{\mathbb{P}^1}\left[\int\limits_{t_{\epsilon}(b)}^{t_{\epsilon}(b)+b\delta_{\epsilon}}
{\left(\Gamma^{\epsilon}_s-\gamma^{\epsilon}_s\right)^2 ds}\right]\\
&= & \mathbb{E}_{\mathbb{P}}\left[\frac{d\mathbb{P}^1}{d\mathbb{P}}\int\limits_{t_{\epsilon}(b)}^{t_{\epsilon}(b)+b\delta_{\epsilon}}
{\left(\Gamma^{\epsilon}_s-\gamma^{\epsilon}_s\right)^2 ds}\right].
\end{eqnarray*}
By Cauchy-Schwarz's inequality  and the mean-one Dol\'eans exponential martingale
property, we have
{\small{
\begin{eqnarray*}
\mathbb{E}_{\mathbb{P}}\left[\frac{d\mathbb{P}^1}{d\mathbb{P}}\int\limits_{t_{\epsilon}(b)}^{t_{\epsilon}(b)+b\delta_{\epsilon}}{\left(\Gamma^{\epsilon}_s-\gamma^{\epsilon}_s\right)^2 ds}\right] &\leq &
\sqrt{
\mathbb{E}_{\mathbb{P}}\left[\exp\left\{\int\limits_{t_{\epsilon}(b)}^{t_{\epsilon}(b)+b\delta_{\epsilon}}{\left(\gamma^{\epsilon}_s\right)^2ds}\right\}
\left(\int\limits_{t_{\epsilon}(b)}^{t_{\epsilon}(b)+b\delta_{\epsilon}}
{\left(\Gamma^{\epsilon}_s-\gamma^{\epsilon}_s\right)^2 ds}\right)^2\right]
}\\
&\leq &
\sqrt{
\mathbb{E}_{\mathbb{P}}
\left[\exp\left\{2\int\limits_{t_{\epsilon}(b)}^{t_{\epsilon}(b)+b\delta_{\epsilon}}{\left(\gamma^{\epsilon}_s\right)^2ds}\right\}\right]}\times\\
&&
\sqrt{
\mathbb{E}_{\mathbb{P}}
\left[\left(\int\limits_{t_{\epsilon}(b)}^{t_{\epsilon}(b)+b\delta_{\epsilon}}
{\left(\Gamma^{\epsilon}_s-\gamma^{\epsilon}_s\right)^2 ds}\right)^4\right]}
\end{eqnarray*}
}}It follows for $\epsilon>0$ small enough that
\begin{eqnarray*}
\exp\left\{\int\limits_{t_{\epsilon}(b)}^{t_{\epsilon}(b)+b\delta_{\epsilon}}{\left(\gamma^{\epsilon}_s\right)^2ds}\right\}&\leq & \exp\left\{
4b\kappa^2_2 \delta_{\epsilon}\left(\kappa_2 \left(t_{\epsilon}(b)+b\delta_{\epsilon}\right)^2+1\right)\sup\limits_{t\in I^{\epsilon}(b)}{B^2_t}+
2bc\kappa^2_2 \delta_{\epsilon}
\right\},
\end{eqnarray*}
where the last expression is $\mathbb{P}$-integrable for $\epsilon>0$ small enough.
Using the scaling property of Brownian motion and the distribution of the maximum of the Brownian motion in a compact interval, the last inequality implies that
\begin{eqnarray*}
\lim\limits_{\epsilon\rightarrow 0}
\mathbb{E}_{\mathbb{P}}
\left[
\exp\left\{\rho\int\limits_{t_{\epsilon}(b)}^{t_{\epsilon}(b)+b\delta_{\epsilon}}{\left(\gamma^{\epsilon}_s\right)^2ds}\right\}\right]&=&1.
\end{eqnarray*}
for any constant $\rho>0$.
Also, it is true that
\begin{eqnarray*}
\left(\int\limits_{t_{\epsilon}(b)}^{t_{\epsilon}(b)+b\delta_{\epsilon}}
{\left(\Gamma^{\epsilon}_s-\gamma^{\epsilon}_s\right)^2 ds}\right)^4 &\leq & \left(b\delta_{\epsilon}\sup\limits_{s\in I^{\epsilon}(b)}{\left(\Gamma^{\epsilon}_s-\gamma^{\epsilon}_s\right)^2}\right)^4\\
&\leq & Cb^4\delta^4_{\epsilon}\left(\sup\limits_{s\in I^{\epsilon}(b)}{\frac{\left(x^{\epsilon}_s-\psi_s\right)^{16}}{{\epsilon^4}}}
+\sup\limits_{s\in I^{\epsilon}(b)}{\frac{\left|x^{\epsilon}_s-\psi_s-\sqrt{\epsilon}y_t\right|^8}{{\epsilon^4}}}
\right),
\end{eqnarray*}
where $C=C(\kappa_2,\kappa_3)>0$ is a constant.
Using the last inequality and Proposition \ref{order}, we obtain that
\begin{eqnarray*}
\lim\limits_{\epsilon\rightarrow 0}\mathbb{E}_{\mathbb{P}}
\left[\left(\int\limits_{t_{\epsilon}(b)}^{t_{\epsilon}(b)+b\delta_{\epsilon}}
{\left(\Gamma^{\epsilon}_s-\gamma^{\epsilon}_s\right)^2 ds}\right)^4\right]&=&0.
\end{eqnarray*}
Consequently,
\begin{eqnarray*}
\lim\limits_{\epsilon\rightarrow 0}{\norm{
x^{\epsilon}_{b\delta_{\epsilon}}\left(x^{\epsilon}_{t_{\epsilon}(b)}\right)-z^{\epsilon}_{b\delta_{\epsilon}}\left(x^{\epsilon}_{t_{\epsilon}(b)}\right)}}&=&0.
\end{eqnarray*}
\end{proof}
Now we have all the tools in order to prove our result for the class of smooth coercive regular potentials.
\begin{theorem}[Smooth Coercive Regular Potentials]\label{mainly}
Let $V$ a smooth coercive regular potential.
Let us consider the family $x^{\epsilon}=\{ {x^{\epsilon}_t} \}_{t \geq 0}$
given by the the semiflow of the following stochastic differential equation,
\begin{eqnarray*}
d{x^{\epsilon}_t}&=&-V^{\prime}({x^{\epsilon}_t})dt+\sqrt{\epsilon}dW_t,\\
{x^{\epsilon}_0}&=&x_0
\end{eqnarray*}
for $t \geq 0$, where $x_0$ is a deterministic initial condition in $\mathbb{R}\setminus \{0\}$ and
$\{W_t\}_{t \geq 0}$ is a standard Brownian motion. This family presents profile cut-off  in the sense of
the Definition \ref{profile} with respect to the total variation distance when $\epsilon$ goes to zero.
The profile function $G:\mathbb{R}\rightarrow \mathbb{R}$ is given by
\begin{eqnarray*}
G(b)&:=&\norm{\N{\left(\tilde{c} e^{-b},1\right)}-\N{\left(0,1\right)}},
\end{eqnarray*}
where $\tilde{c}$ is the nonzero constant given by
\begin{eqnarray*}
\lim\limits_{t\rightarrow +\infty}{e^{V^{\prime\prime}{(0)}t}\psi_t}&=&:\tilde{c}.
\end{eqnarray*}
and the cut-off time $t_{\epsilon}$ and window time $w_{\epsilon}$ are given by
\begin{eqnarray*}
t_{\epsilon}&:=&\frac{1}{2V^{\prime\prime}(0)}\left(\ln\left(\frac{1}{\epsilon}\right)+\ln\left(2V^{\prime\prime}(0)\right) \right),\\
w_{\epsilon}&:=&\frac{1}{V^{\prime\prime}(0)}+\delta_{\epsilon},
\end{eqnarray*}
where $\delta_\epsilon=\epsilon^{\gamma}$ for some $\gamma\in ]0,1[$.
\end{theorem}
\begin{proof}
Let $\epsilon>0$ and $t>0$ be fixed. We define
\begin{eqnarray*}
D^{\epsilon}{(t)}:=\norm{\x-\mu^{\epsilon}}
\end{eqnarray*}
and
\begin{eqnarray*}
d^{\epsilon}{(t)}:=\norm{z^{\epsilon}_t-\N^{\epsilon}},
\end{eqnarray*}
where $\mu^{\epsilon}$ and $\N^{\epsilon}$ are given by  Lemma
\ref{density} and Lemma \ref{invariantes}. For each $b\in \R$, take $\epsilon_b>0$ such that
${t}^*_{\epsilon}(b):=t_{\epsilon}+b(w_{\epsilon}+\delta_{\epsilon})={t}_{\epsilon}(b)+
b\delta_{\epsilon}> 0$ for every $0<\epsilon<\epsilon_b$.
By Corollary \ref{first} and Remark \ref{short} we know that for each $b \in \R$
\begin{eqnarray}\label{cutlin}
\lim\limits_{\epsilon\rightarrow 0}{d^{\epsilon}\left({t}^*_{\epsilon}(b)\right)}&=&G(b).
\end{eqnarray}
By definition,
\begin{eqnarray*}
D^{\epsilon}{({t}^*_{\epsilon}(b))}&=&\norm{x^{\epsilon}_{{t}^*_{\epsilon}(b)}-\mu^{\epsilon}}\\
&\leq&{\norm{
x^{\epsilon}_{b\delta_{\epsilon}}\left(x^{\epsilon}_{t_{\epsilon}(b)}\right)-z^{\epsilon}_{b\delta_{\epsilon}}\left(x^{\epsilon}_{t_{\epsilon}(b)}\right)}}+
\norm{
z^{\epsilon}_{b\delta_{\epsilon}}\left(x^{\epsilon}_{t_{\epsilon}(b)}\right)-z^{\epsilon}_{b\delta_{\epsilon}}\left(z^{\epsilon}_{t_{\epsilon}(b)}\right)}+\\
&&\norm{z^{\epsilon}_{{t}^*_{\epsilon}(b)}-\N^{\epsilon}}+\norm{\N^{\epsilon}-\mu^{\epsilon}}.
\end{eqnarray*}
Using Proposition \ref{lincou}, Proposition \ref{shorttime}, Lemma \ref{density}, the relation (\ref{cutlin}) and the item $i)$ of Lemma \ref{supinf}, we have $\limsup\limits_{\epsilon\rightarrow 0}{D^{\epsilon}{({t}^*_{\epsilon}(b))}}\leq G(b)$. In order to obtain the converse inequality, we observe that
\begin{eqnarray*}
d^{\epsilon}{({t}^*_{\epsilon}(b))}&=&\norm{z^{\epsilon}_{{t}^*_{\epsilon}(b)}-\N^{\epsilon}}\\
&\leq &\norm{
z^{\epsilon}_{b\delta_{\epsilon}}\left(z^{\epsilon}_{t_{\epsilon}(b)}\right)-z^{\epsilon}_{b\delta_{\epsilon}}\left(x^{\epsilon}_{t_{\epsilon}(b)}\right)}+
{\norm{
z^{\epsilon}_{b\delta_{\epsilon}}\left(x^{\epsilon}_{t_{\epsilon}(b)}\right)-x^{\epsilon}_{b\delta_{\epsilon}}\left(x^{\epsilon}_{t_{\epsilon}(b)}\right)}}+\\
&&\norm{x^{\epsilon}_{{t}^*_{\epsilon}(b)}-\mu^{\epsilon}}+\norm{\mu^{\epsilon}-\N^{\epsilon}}.
\end{eqnarray*}
Again, using Proposition \ref{lincou}, Proposition \ref{shorttime}, Lemma \ref{density}, the relation (\ref{cutlin}) and the item $ii)$ of Lemma \ref{supinf} we have $\liminf\limits_{\epsilon\rightarrow0}{D^{\epsilon}{({t}^*_{\epsilon}(b))}}\geq G(b)$.
Consequently,
\begin{eqnarray*}
\lim\limits_{\epsilon\rightarrow0}{D^{\epsilon}{({t}^*_{\epsilon}(b))}}=G(b).
\end{eqnarray*}
\end{proof}

The following proposition will permit us to approximate a coercive regular potential by a smooth coercive regular potential.

\begin{proposition}[Removing Boundedness for $V^{\prime\prime}$ and $V^{\prime\prime\prime}$]\label{baprox}
Let us assume that $V$ is a coercive regular potential.
For every $M\in ]0,+\infty[$, there exists
 a smooth coercive regular potential $V_M(x)$ which is an approximation of $V$ in the following way:
$V_M(x)=V(x)$ for every $|x|\leq \sqrt{2} M$.
\end{proposition}
\begin{proof}
By coercivity hypothesis there exists $\delta>0$ such that
$V^{\prime\prime}(x)\geq \delta$ for every $x\in \R$.
Let $g \in\mathcal{C}^{\infty}\left(\mathbb{R},[0,1]\right)$ be an increasing function
such that $g(u)=0$ for $u\leq \frac{1}{2}$ and $g(u)=1$ if $u\geq 1$. Let $M\in[1,\infty[$ be a fixed number.
Let $R_M:\mathbb{R}\rightarrow \mathbb{R}$ be a function defined by
\begin{eqnarray*}
R_M(x)&=&g\left(\frac{x^2}{4M^2}\right)\delta+\left(1-g\left(\frac{x^2}{4M^2}\right)\right)V^{\prime\prime}(x).
\end{eqnarray*}
Since $V\in\mathcal{C}^3\left(\R,\R\right)$ and $g \in\mathcal{C}^{\infty}\left(\mathbb{R},[0,1]\right)$,
we have $R_M\in\mathcal{C}^{1}\left(\mathbb{R},\R\right)$.
We also have that $R_M(x)=V^{\prime \prime}(x)$ for every $|x|\leq \sqrt{2}M$, $R_M(x)=\delta$ for every $|x|\geq 2M$,
$R_M(x)\geq \delta$ for every $x\in \mathbb{R}$, $\|R_M\|_{\infty}<+\infty$ and $\|R^{\prime}_M\|_{\infty}<+\infty$.
Let us  define $S_M(x):=\int\limits_{0}^{x}{R_M(y)dy}$ for every $x\in \mathbb{R}$ and let us define
$V_M(x):=\int\limits_{0}^{x}{S_M(y)dy}$.
Then $V_M$ is a smooth $\delta$-coercive regular potential such that
$V_M(x)=V(x)$ for every $|x|\leq \sqrt{2}M$.
\end{proof}

The next proposition will tell us that the approximation of the coercive regular potential by a smooth coercive regular potential also implies an approximation in the total variation distance of the invariant measures associated to the potential $V$ and $V_M$
and the total variation distance for the processes at the ``cut-off time" associated to the potentials $V$ and $V_M$.
\begin{proposition}\label{rb}
Let $V$ be a coercive regular potential and for every $M>0$ let $V_M$ be the approximation of $V$ obtained from Proposition \ref{baprox}.
Let $x^{\epsilon,M}=\left\{x^{\epsilon,M}_{t}\right\}_{t\geq 0}$ be the It\^o diffusion associated to the smooth coercive potential $V_M$
and let $\mu^{\epsilon,M}$ be the invariant probability measure associated to the stochastic process
$x^{\epsilon,M}$ defined in Lemma \ref{density}.
Let us denote by $x^{\epsilon}=\left\{x^{\epsilon}_{t}\right\}_{t\geq 0}$ the It\^o diffusion associated to the coercive potential $V$
and let us denote by $\mu^{\epsilon}$ the invariant probability measure associated to the stochastic process
$x^{\epsilon}$ defined in Lemma \ref{density}.

It follows that
\begin{itemize}
\item[$i)$] For every $M>0$
\begin{eqnarray*}
\lim\limits_{\epsilon\rightarrow 0}{\norm{\mu^{\epsilon}-\mu^{\epsilon,M}}}
&=&
0
\end{eqnarray*}
\item[$ii)$]
Using the same notation as in Theorem \ref{mainly}, for each $b\in \R$, take $\epsilon_b>0$ such that
${t}^*_{\epsilon}(b):=t_{\epsilon}+b(w_{\epsilon}+\delta_{\epsilon})={t}_{\epsilon}(b)+
b\delta_{\epsilon}> 0$ for every $0<\epsilon<\epsilon_b$, where $\delta_\epsilon=\epsilon^\gamma$ for some $\gamma>0$. 
\begin{eqnarray*}
 {\lim\limits_{\epsilon\rightarrow 0}{\norm{x^{\epsilon}_{{t}^*_{\epsilon}(b)}-x^{\epsilon,M}_{{t}^*_{\epsilon}(b)}}}}&=&0
\end{eqnarray*}
for every $M>|x_0|$ and every $b\in \R$.
\end{itemize}
\end{proposition}

\begin{proof}
Let us prove item $i)$.
Notice that $V''_M(0)=V''(0)$. By triangle's inequality and Lemma \ref{lemma1}, we have
\begin{eqnarray*}
{\norm{\mu^{\epsilon}-\mu^{\epsilon,M}}}&\leq & {\norm{\mu^{\epsilon}-\N^{\epsilon}}}
+{\norm{\N^{\epsilon}-\mu^{\epsilon,M}}}.\\
\end{eqnarray*}
Taking $\epsilon\rightarrow 0$ and using Lemma \ref{invariantes} we obtain
\begin{eqnarray*}
\lim\limits_{\epsilon\rightarrow 0}{\norm{\mu^{\epsilon}-\mu^{\epsilon,M}}}
&=&
0\end{eqnarray*}
for every $M>0$.
Now let us prove item $ii)$. Let $\epsilon>0$ and $M>|x_0|>0$ be fixed.
Let us define
$\tau^{\epsilon,M}:=\inf\left\{s\geq 0: \left|x^{\epsilon,M}_s\right|> M\right\}$. By the variational definition of total variation distance in terms of couplings
\begin{eqnarray*}
{{\norm{x^{\epsilon}_{{t}^*_{\epsilon}(b)}-x^{\epsilon,M}_{{t}^*_{\epsilon}(b)}}}}&\leq & {
{\mathbb{P}_{x_0}\left(\tau^{\epsilon,M}\leq {t}^*_{\epsilon}(b) \right)}}.
\end{eqnarray*}
Let us define $\sigma^{\epsilon,M}:=\inf\left\{s\geq 0: |x^{\epsilon,M}_s-\psi^{M}_s|>M-|x_0| \right\}$, where
$\psi^{M}:=\left\{\psi^{M}_t\right\}$
is the semiflow associated to the autonomous differential equation,
\begin{eqnarray*}
d\psi^{M}_t=-V_M'\left(\psi^{M}_t\right)
\end{eqnarray*} for every $t\geq 0$ and $\psi^{M}_0:=x_0$.
Using the coercivity hypothesis of $V_M$, we see that the semiflow $\psi^M$ is decreasing in norm, and $|\psi^{M}_t|\leq |x_0|$ for every $t\geq 0$.
In particular,
$\sigma^{\epsilon,M}\leq \tau^{\epsilon,M}$. Consequently.
$\mathbb{P}_{x_0}\left(\tau^{\epsilon,M}\leq {t}^*_{\epsilon}(b) \right)\leq \mathbb{P}_{x_0}\left(\sigma^{\epsilon,M}\leq {t}^*_{\epsilon}(b) \right)$.

Therefore, it is enough to prove that
$\lim\limits_{\epsilon\rightarrow 0}\mathbb{P}_{x_0}\left(\sigma^{\epsilon,M}> {t}^*_{\epsilon}(b) \right)=1$.
For every $s\geq 0$, let us define $z^{\epsilon,M}_s:=\frac{x^{\epsilon,M}_s-\psi^M_s}{\sqrt{\epsilon}}$.
Then, $\sigma^{\epsilon,M}=\inf\left\{s\geq 0: |z^{\epsilon,M}_s|>\frac{M-|x_0|}{\sqrt{\epsilon}} \right\}$.
We note that
\begin{eqnarray*}
\mathbb{P}_{x_0}\left(\sigma^{\epsilon,M}\geq {t}^*_{\epsilon}(b) \right)&=&\mathbb{P}_{x_0}\left(\sup\limits_{0\leq s \leq  {t}^*_{\epsilon}(b)}{\left|z^{\epsilon,M}_s\right|}\leq \frac{M-|x_0|}{\sqrt{\epsilon}}   \right).
\end{eqnarray*}
Let us define $c_M:=M-|x_0|>0$. We have
\begin{eqnarray*}
\mathbb{P}_{x_0}\left(\sup\limits_{0\leq s \leq  {t}^*_{\epsilon}(b)}{\left|z^{\epsilon,M}_s\right|}> \frac{c_M}{\sqrt{\epsilon}}   \right)&=&
\mathbb{P}_{x_0}\left(\sup\limits_{0\leq s \leq  {t}^*_{\epsilon}(b)}{\left(z^{\epsilon,M}_s\right)^2}> \frac{c^2_M}{\epsilon}   \right).
\end{eqnarray*}
Using It\^o's formula and the coercivity of $V_M$, we have
\begin{eqnarray*}
\left(z^{\epsilon,M}_t\right)^2\leq t+\Pi^{\epsilon,M}_t
\end{eqnarray*}
for every $t\geq 0$, where the process $\Pi^{\epsilon,M}_t:=2\int\limits_{0}^{t}{z^{\epsilon,M}_sdW_s}$ is a martingale. Then,
\begin{eqnarray*}
\mathbb{E}\left[\left(z^{\epsilon,M}_t\right)^2\right] \leq t
\end{eqnarray*}
for every $t\geq 0$. Using It\^o's isometry, we obtain
\begin{eqnarray*}
\mathbb{E}\left[\left(\Pi^{\epsilon,M}_t\right)^2\right] \leq 2t^2
\end{eqnarray*}
for every $t\geq 0$. Let us take $\epsilon_{M,b}>0$ such that for every $0<\epsilon<\epsilon_{M,b}$, we have
$c^2_M-\epsilon {t}^*_{\epsilon}(b)>0$.
Using Doob's inequality, we have
\begin{eqnarray*}
\mathbb{P}_{x_0}\left(\sup\limits_{0\leq s \leq  {t}^*_{\epsilon}(b)}{\left(z^{\epsilon,M}_s\right)^2}> \frac{c^2_M}{\epsilon}   \right)&\leq &
\mathbb{P}_{x_0}\left(\sup\limits_{0\leq s \leq  {t}^*_{\epsilon}(b)}{\left|\Pi^{\epsilon,M}_s\right|}>
\frac{c^2_M-\epsilon {t}^*_{\epsilon}(b)}{\epsilon}   \right)\\
& \leq &
\frac{\epsilon^2}{\left(c^2_M-\epsilon {t}^*_{\epsilon}(b)\right)^2}\mathbb{E}\left[\left(\Pi^{\epsilon,M}_{{t}^*_{\epsilon}(b)}\right)^2\right]\\
& \leq &
\frac{2\epsilon^2 \left({t}^*_{\epsilon}(b)\right)^2}{\left(c^2_M-\epsilon {t}^*_{\epsilon}(b)\right)^2}.
\end{eqnarray*}
Letting $\epsilon\rightarrow 0$ we obtain the desired limit.
\end{proof}

Now, we are ready to prove Theorem \ref{main}. To stress the fact that Theorem \ref{main} is just a consequence of what we have proved up to here, let us state this as a Lemma:

\begin{lemma}[From Smooth Coercive Case to the General Case]\label{acoplamiento}
Let $V_M$ be the approximation of $V$ obtained in Proposition \ref{baprox}.
Profile cut-off  for $\{x^{\epsilon,M}_t\}_{t\geq 0}$ implies
profile cut-off  for $\{x^{\epsilon}_t\}_{t\geq 0}$ with the same cut-off time, cut-off window and profile function.
\end{lemma}

\begin{proof}
Recall the notation introduced in Proposition  \ref{lincou}. Let $\epsilon>0$ and $t>0$ be fixed.
Let us take $M>\max\left\{|x_0|,\|\psi\|_{\infty}\right\}$. We define
\begin{eqnarray*}
D^{\epsilon,M}{(t)}:=\norm{x^{\epsilon,M}_{t}-\mu^{\epsilon,M}}
\end{eqnarray*} and
\begin{eqnarray*}
D^{\epsilon}{(t)}:=\norm{x^{\epsilon}_{t}-\mu^{\epsilon}}.
\end{eqnarray*}
By triangle's inequality, we have
\begin{eqnarray*}
D^{\epsilon,M}{(t)}&\leq &  \norm{x^{\epsilon,M}_t-x^{\epsilon}_t}+D^{\epsilon}(t)+\norm{\mu^{\epsilon}-\mu^{\epsilon,M}}.
\end{eqnarray*}
Recall that
$t^{}_{\epsilon}=\frac{1}{2V^{\prime\prime}(0)}\left(\ln\left(\frac{1}{\epsilon}\right)+\ln\left(2V^{\prime\prime}(0)\right) \right)$
and
$w^{\prime}_{\epsilon}=\frac{1}{V^{\prime\prime}(0)}+\delta_{\epsilon}$ respectively.
Let $b\in \R$ be fixed. Recall that
${t}^*_{\epsilon}(b)=t_{\epsilon}+bw^\prime_{\epsilon}$. Take $\epsilon_b>0$ such that for every $0<\epsilon<\epsilon_b$ we have,
${t}^*_{\epsilon}(b)>0$.
Consequently,
\begin{eqnarray*}
D^{\epsilon,M}{({t}^*_{\epsilon}(b))}
&\leq & \norm{x^{\epsilon,M}_{{t}^*_{\epsilon}(b)}-x^{\epsilon}_{{t}^*_{\epsilon}(b)}}+D^{\epsilon}({t}^*_{\epsilon}(b))+\norm{\mu^{\epsilon}-\mu^{\epsilon,M}}.
\end{eqnarray*}
Therefore, using Proposition \ref{rb} and Lemma \ref{supinf} we have
\begin{eqnarray*}
\limsup\limits_{\epsilon\rightarrow 0}{D^{\epsilon,M}{({t}^*_{\epsilon}(b))}}&\leq &  \limsup\limits_{\epsilon\rightarrow 0}{D^{\epsilon}({t}^*_{\epsilon}(b))}.
\end{eqnarray*}
By Theorem \ref{mainly}, we know that
$\lim\limits_{\epsilon\rightarrow 0}{D^{\epsilon,M}{({t}^*_{\epsilon}(b))}}=G(b)$. Therefore
\begin{eqnarray*}
G(b)&\leq &  \limsup\limits_{\epsilon\rightarrow 0}{D^{\epsilon}({t}^*_{\epsilon}(b))}.
\end{eqnarray*}
It also follows that
\begin{eqnarray*}
D^{\epsilon}{({t}^*_{\epsilon}(b))}&\leq &  \norm{x^{\epsilon}_{{t}^*_{\epsilon}(b)}-x^{\epsilon,M}_{{t}^*_{\epsilon}(b)}}+D^{\epsilon,M}({t}^*_{\epsilon}(b))+\norm{\mu^{\epsilon,M}-\mu^{\epsilon}}.
\end{eqnarray*}
Therefore, using Lemma \ref{supinf}, Proposition \ref{rb} and Theorem \ref{mainly} we have
\begin{eqnarray*}
\liminf\limits_{\epsilon\rightarrow 0}{D^{\epsilon}{({t}^*_{\epsilon}(b))}}&\leq & G(b).
\end{eqnarray*}
We conclude that
\begin{eqnarray*}
\lim\limits_{\epsilon\rightarrow 0}{D^{\epsilon}{({t}^*_{\epsilon}(b))}}&=& G(b).
\end{eqnarray*}
\end{proof}

\section{Double Well Potential}\label{dwp}


We study the situation when the potential $V$ has only two wells of different depths.  In this situation we can observe two statistical different regimes.
Firstly, if the horizon is shorter that the exit time from the shallow well, the system cannot leave the well where it has started, and therefore stays in the neighborhood of the well's local minimum.
Secondly, if the time horizon is longer that the exit time from the shallow well, the system has enough time to reach the deepest well from any starting point, and stays in a neighborhood of the global minimum. C. Kipnis and C. Newman in \cite{KN} proved the following metastability behavior: there is a time scale on which the dynamical system converges to a Markov two-state process with one absorbing state corresponding to the deep well. This time scale is given by the mean exit time from the shallow well.

Using the last fact we can observe the following: 
Let us denote by $x{^-}$ the shallow well and by $x^+$ the deepest well.
In Theorem \ref{main} we prove that on the one-well potential case under the coercivity assumption we have profile cut-off phenomenon. 
Given a deterministic initial condition  $x_0$ in a small neighborhood of the well of $x^*$ where $x^* \in \{x^-,x^+\}$. Recall that  $\{x^{\epsilon}_t\}_{t\geq 0}$ is the the following differential equation
\begin{eqnarray*}
d{x^{\epsilon}_t}&=&-V^{\prime}({x^{\epsilon}_t})dt+\sqrt{\epsilon}dW_t,\\
\end{eqnarray*}
for $t \geq 0$. 
Let us suppose that $V^{\prime\prime}(x^*)>0$.
We have that the exit time from the well associated to the local minimum $x^*$ is exponentially large and the time of the cut-off time obtained in the Theorem \ref{main} is much smallest. 
Consequently, we have abrupt convergence to a ``kind local asymptotic distribution"; that is we have a kind of local cut-off phenomenon with respect to the following distance:
\begin{eqnarray*}
d_\epsilon(t):=\norm{x^\epsilon_t-\mathcal{N}\left(0,\frac{\epsilon}{2V^{\prime\prime}(x^*)}\right)},
\end{eqnarray*}
with time cut-off  $t_\epsilon=\frac{1}{2V^{\prime\prime}(x^*)}\ln\left(\frac{1}{\epsilon}\right)$, window cut-off $w_\epsilon=\frac{1}{V^{\prime\prime}(x^*)}+\epsilon^\gamma$ for some 
$\gamma\in ]0,1[$ and profile function $G:\mathbb{R}\rightarrow [0,1]$ given by 
\begin{eqnarray*}
G(b)&:=&\norm{\N{\left(c(x^*) e^{-b},1\right)}-\N{\left(0,1\right)}},
\end{eqnarray*}
where $c(x^*)$ is the nonzero constant given by
\begin{eqnarray*}
\lim\limits_{t\rightarrow +\infty}{e^{V^{\prime\prime}{(x^*)}t}\Phi_t}&=&c(x^*).
\end{eqnarray*}
\begin{remark}
By the same facts, the last local cut-off phenomenon can also extend for a multi-well potential. 
\end{remark}

\appendix\label{appendix}
\section{Properties of the Total Variation Distance of Normal Distribution}\label{ap1}

\begin{lemma}\label{lemma1}
Let $\{\mu,\tilde{\mu}\} \subset \mathbb{R}$ and $\{\sigma^2,\tilde{\sigma}^2\}\subset ]0,+\infty[$ be fixed numbers.
\begin{itemize}
\item[$i)$] For any constant $c\not= 0$ we have
\begin{eqnarray*}
\norm{\N{\left(c\mu,{c}^2\sigma^2 \right)}-\N{\left(c\tilde{\mu},{c}^2\tilde{\sigma}^2\right)}}&=&
\norm{\N{\left(\mu,\sigma^2\right)}-\N{\left(\tilde{\mu},\tilde{\sigma}^2\right)}}.
\end{eqnarray*}
\item[$ii)$]
\begin{eqnarray*}
\norm{\N{\left(\mu,\sigma^2 \right)}-\N{\left(\tilde{\mu},\tilde{\sigma}^2 \right)}}&=&
\norm{\N{\left(|\mu-\tilde{\mu}|,\sigma^2\right)}-\N{\left(0,\tilde{\sigma}^2 \right)}}.
\end{eqnarray*}
\end{itemize}
\end{lemma}
\begin{proof}
This is done using the characterization of the total variation distance between two probability measures which are absolutely continuous with respect to the Lebesgue measure on $\left(\mathbb{R},\mathcal{B}\left(\R\right)\right)$ and using the Change of Variable Theorem.
\end{proof}

\begin{lemma}\label{lemma2}
Let $\mu\in \R$ then
\begin{eqnarray*}
\norm{\N{\left(\mu,1 \right)}-\N{\left(0,1\right)}}=\frac{2}{\sqrt{2\pi}}\int\limits_{0}^{\nicefrac{|\mu|}{2}}{e^{-\frac{x^2}{2}}dx} \leq \frac{|\mu|}{\sqrt{2\pi}}.
\end{eqnarray*}
\end{lemma}
\begin{proof}
Also, this is done using the characterization of the total variation distance between two probability measures which are absolutely continuous with respect to the Lebesgue measure on $\left(\mathbb{R},\mathcal{B}\left(\R\right)\right)$ and an straightforward calculations.
\end{proof}

\begin{lemma}\label{adicional}
Let $\{\mu_{\epsilon}\}_{\epsilon>0}\subset \R$ be a sequence such that $\lim\limits_{\epsilon\rightarrow 0}{\mu_{\epsilon}}=\mu\in \R$. Then
\begin{eqnarray*}
\lim\limits_{\epsilon\rightarrow 0}\norm{\N{\left(\mu_{\epsilon},1 \right)}-\N{\left(0,1\right)}}=\norm{\N{\left(\mu,1\right)}-\N{\left(0,1\right)}}.
\end{eqnarray*}
\end{lemma}
\begin{proof}
This is done using triangle inequality, the item $ii)$ of Lemma \ref{lemma1}, Lemma \ref{lemma2} and the
Lemma \ref{supinf}.
\end{proof}

\begin{lemma}\label{adicionalvar}
Let $\{\sigma^2_{\epsilon}\}_{\epsilon>0}\subset ]0,+\infty[$ be a sequence such that $\lim\limits_{\epsilon\rightarrow 0}{\sigma^2_{\epsilon}}=\sigma^2\in ]0,+\infty[$. Then
\begin{eqnarray*}
\lim\limits_{\epsilon\rightarrow 0}\norm{\N{\left(0,\sigma^2_{\epsilon} \right)}-\N{\left(0,\sigma^2\right)}}&=&0.
\end{eqnarray*}
\end{lemma}
\begin{proof}
This is done using the item $i)$ of Lemma \ref{lemma1}, the characterization of the total variation distance between two probability measures which are absolutely continuous with respect to the Lebesgue measure on $\left(\mathbb{R},\mathcal{B}\left(\R\right)\right)$
and an straightforward calculations.
\end{proof}

\begin{lemma}[Total Variation Bounded]\label{conditional}
Let $(\Omega,\mathcal{F},\mathbb{P})$ be a probability space and $\mathcal{G}\subset \mathcal{F}$ be a sub-sigma algebra of $\mathcal{F}$.
Let $X,Y,Z:(\Omega,\mathcal{F})\rightarrow (\mathbb{R},\mathcal{B}(\mathbb{R}))$ be random variables such that $X$ and $Y$ are $\mathcal{G}$ measurables and $X,Y,Z\in L^{1}\left(\Omega,\mathcal{F},\mathbb{P}\right)$. Let us consider the following random variables $X^{*}=X+Z$ and $Y^{*}=Y+Z$. Let us suppose that for some $\sigma^2>0$ we have
$\mathbb{P}\left[X^{*}\in F\left|\right.\mathcal{G}\right]=\mathbb{P}\left[\mathcal{G}(X,\sigma^2)\in F\right]$ and
$\mathbb{P}\left[Y^{*}\in F\left|\right.\mathcal{G}\right]=\mathbb{P}\left[\mathcal{G}(Y,\sigma^2)\in F\right]$
 for every $F\in \mathcal{F}$.
Then
\begin{eqnarray*}
\norm{X^{*}-Y^{*}}&\leq & \frac{1}{\sqrt{2\pi}\sigma}\mathbb{E}{\left[\left|X-Y\right|\right]}.
\end{eqnarray*}
\end{lemma}

\begin{proof}
Using the the properties of conditional expectation, the item $i)$, item $ii)$ of Lemma \ref{lemma1} and Lemma \ref{lemma2}, we have
\begin{eqnarray*}
\norm{X^{*}-Y^{*}}&=&\sup\limits_{F\in \mathcal{F}}
{\left|\mathbb{E}\left[\mathbbm{1}_{\left(X^* \in F \right)}-\mathbbm{1}_{\left(Y^* \in F\right)} \right]\right|}\\
&\leq &
\sup\limits_{F\in \mathcal{F}}
{\mathbb{E}\left[ \left| \mathbb{E}\left[\mathbbm{1}_{\left(X^* \in F\right)}-\mathbbm{1}_{\left(Y^* \in F\right)}\left|\right.\mathcal{G} \right] \right|  \right]}\\
&\leq &
\sup\limits_{F\in \mathcal{F}}
{\mathbb{E}\left[ \left| \mathbb{P}\left(\N(X,\sigma^2)\in F\right)-\mathbb{P}\left(\N(Y,\sigma^2)\in F\right) \right|  \right]}\\
&\leq &
\sup\limits_{F\in \mathcal{F}}
{\mathbb{E}\left[ \frac{1}{\sqrt{2\pi}\sigma}\left| X-Y \right|  \right]}\\
&=&
\frac{1}{\sqrt{2\pi}\sigma}\mathbb{E}{\left[\left|X-Y\right|\right]}.
\end{eqnarray*}
\end{proof}

\section{Qualitative and Quantitative Behavior}\label{ap2}
\begin{lemma}\label{importante0}
Let us assume the hypothesis of Theorem \ref{toy}.
 Then
\begin{itemize}
\item[$i)$]  $\lim\limits_{t\rightarrow +\infty}{\psi_t}=0$.
\item[$ii)$] $\lim\limits_{t\rightarrow +\infty}{\Phi_t}=0$.
\item[$iii)$] There exist constants $c\not=0$ and $\tilde{c}\not=0$ such that
\begin{eqnarray*}
\lim\limits_{t\rightarrow +\infty}{e^{V^{\prime\prime}{(0)}t}\Phi_t}&=&c,
\end{eqnarray*}
\begin{eqnarray*}
\lim\limits_{t\rightarrow +\infty}{e^{V^{\prime\prime}{(0)}t}\psi_t}&=&\tilde{c},
\end{eqnarray*}
where $\Phi=\{\Phi_{t}\}_{t\geq 0}$ is the fundamental solution of the non-autonomous system
\begin{eqnarray*}
d\Phi_{t}=-V^{\prime\prime}{(\psi_t)}\Phi_t dt
\end{eqnarray*}
for every $t\geq 0$ with initial condition $\Phi_{0}=1$.
\item[$iv)$] $\lim\limits_{t\rightarrow +\infty}{\Phi_t^2 \int\limits_{0}^{t}{\left(\frac{1}{\Phi_s}\right)^2d{{s}}}}=\frac{1}{2V^{\prime\prime}{(0)}}$.
\end{itemize}
\end{lemma}

\begin{proof} ~
\begin{itemize}
\item[$i)$] By our assumptioms $V^{\prime}(0)=0$, $V^{\prime\prime}(0)>0$ and $V^{\prime}(x)\not=0$ if $x\not=0$. Therefore, the unique critical point zero is
asymptotically stable, so there exists an open neighboorhood $N_0$ of zero such that for every $\psi_0\in N_0$. It follows that
$\psi_t$ goes to zero as $t$ goes to infinity.
Let us consider that $\psi_0\not \in N_{0}$ and $K:=V^{-1}\left([0,V(\psi_0)]\right)$. Then $\psi_t\in K$ for every $t\geq 0$.
Also, $K$ is a compact set because of $\lim\limits_{|x|\rightarrow +\infty}{V(x)}=+\infty$.
Because $K$ is bounded, then
there exist $r>0$ such that $K\subset B(0,r)$ where we denote $B(0,r):=\{x\in \mathbb{R}: |x|<r\}$
and $\overline{B(0,r)}:=\{x\in \mathbb{R}: |x|\leq r\}$ so we
we can choose $N_0$ small enough such that $N_0\subset B(0,r)\subset \overline{B(0,r)}$. Let us
call $\hat{K}:=\overline{B(0,r)}$ then $\psi_t \in \hat{K}$ for every $t\geq 0$.
Let us define $\delta:=\inf\limits_{x\in \hat{K}\setminus N_0}{\left(V^{\prime}(x)\right)^2}>0$.
Let us suppose that
$\psi_t\not\in N_{0}$ for every $t\geq 0$, then $dV(\psi_t)=-\left(V^{\prime}(\psi_t)\right)^2\leq -\delta$ for every $t\geq 0$. Therefore, $0\leq t\leq \frac{V(\psi_0)}{\delta}$ which is a contradiction. Consequently, there exists $\tau>0$ such that $\psi_{\tau}\in N_0$ and consequently, $\psi_t$ goes to zero as $t$ goes to infinity.

\item[$ii)$] By our assumptions it follows that $\Phi_t=\frac{V^{\prime}(\psi_t)}{V^{\prime}(\psi_0)}$ for every $t\geq 0$, where $\psi_0=x_0\not=0$. So by item $i)$ and continuity of $V^{\prime}$ we have
$\lim\limits_{t\rightarrow \infty}{\Phi_t}=\frac{V^{\prime}(0)}{V^{\prime}(\psi_0)}=0$.

\item[$iii)$] Let us define $H(z)=\left(\frac{V^{\prime\prime}(0)}{V^{\prime}(z)}-\frac{1}{z}\right)\mathbbm{1}_{\{z\not=0\}}
+\left(\frac{-V^{\prime\prime\prime}(0)}{2V^{\prime\prime}(0)}\right)\mathbbm{1}_{\{z=0\}}$,
where $\mathbbm{1}_{A}$ denotes the indicator function of the set $A\subset \R$.
Let us define $h:\R\rightarrow \R$ by
\begin{eqnarray*}
h(x):=x\exp{\left(\int\limits_{0}^{x}{H(z)dz}\right)}.
\end{eqnarray*}
Since $H$ is everywhere continuous, then it follows that $h$ is well defined.
Let us define $\Psi_t:=h(\psi_t)$ for every $t\geq 0$, then $d\Psi_t=-V^{\prime\prime}{(0)}\Psi_t dt$ for every $t\geq 0$
and $\Psi_0=h(\psi_0)$. Therefore,
\begin{eqnarray}\label{ayuda}
\psi_t \exp{\left(V^{\prime\prime}(0)t\right)}&=&h(\psi_0)\exp{\left(-\int\limits_{0}^{\psi_t}{H(z)dz}\right)}
\end{eqnarray}
for every $t\geq 0$. By Intermediate Value Theorem, for every $t\geq 0$ there exists $\xi_t\in ]\min\{0,\psi_t\},\max\{0,\psi_t\}[$
such that $V^{\prime}(\psi_t)=V^{\prime\prime}(\xi_t)\psi_t$. Because of relation (\ref{ayuda}), we see that
\begin{eqnarray}\label{ayuda2}
V^{\prime}(\psi_t) \exp{\left(V^{\prime\prime}(0)t\right)}&=&V^{\prime\prime}(\xi_t)h(\psi_0)\exp{\left(-\int\limits_{0}^{\psi_t}{H(z)dz}\right)}
\end{eqnarray}
for every $t\geq 0$. Consequently, by the relation (\ref{ayuda2}) and item $ii)$, we have
\begin{eqnarray*}
\lim\limits_{t\rightarrow +\infty}{V^{\prime}(\psi_t) \exp{\left(V^{\prime\prime}(0)t\right)}}&=&
V^{\prime\prime}(0)h(\psi_0).
\end{eqnarray*}
Because  $sgn(h(x))=sgn(x)$ for every $x\not=0$, then
$V^{\prime\prime}(0)h(\psi_0)\not=0$.

\item[$iv)$] By item $ii)$, we have
\begin{eqnarray*}
{\Phi_t^2 \int\limits_{0}^{t}{\left(\frac{1}{\Phi_s}\right)^2d{{s}}}}&=&
{(V^{\prime}(\psi_t))^2 \int\limits_{0}^{t}{\left(\frac{1}{V^{\prime}(\psi_s)}\right)^2d{{s}}}}
\end{eqnarray*}
for each $t\geq 0$. By item $iii)$ and for each $0<\epsilon<c^2$, we have
\begin{eqnarray*}
\limsup\limits_{t\rightarrow +\infty}{(V^{\prime}(\psi_t))^2 \int\limits_{0}^{t}{\left(\frac{1}{V^{\prime}(\psi_s)}\right)^2d{{s}}}}
&\leq & \left(\frac{c^2+\epsilon}{c^2-\epsilon}\right) \frac{1}{2V^{\prime\prime}(0)},\\
\liminf\limits_{t\rightarrow +\infty}{(V^{\prime}(\psi_t))^2 \int\limits_{0}^{t}{\left(\frac{1}{V^{\prime}(\psi_s)}\right)^2d{{s}}}}
&\geq  & \left(\frac{c^2-\epsilon}{c^2+\epsilon}\right) \frac{1}{2V^{\prime\prime}(0)}.\\
\end{eqnarray*}
Letting $\epsilon\rightarrow 0$, we obtain
\begin{eqnarray*}
\lim\limits_{t\rightarrow +\infty}{(V^{\prime}(\psi_t))^2 \int\limits_{0}^{t}{\left(\frac{1}{V^{\prime}(\psi_s)}\right)^2d{{s}}}}
&= & \frac{1}{2V^{\prime\prime}(0)}.\\
\end{eqnarray*}
\end{itemize}
\end{proof}

\begin{lemma}\label{porque}
Let us assume the hypothesis of Theorem \ref{toy}. Let us follow the same notation as in the proof of Theorem \ref{mainly}.
It follows that
\begin{eqnarray*}\lim\limits_{\epsilon\rightarrow 0}{\frac{\sup\limits_{{t}_{\epsilon}(b)\leq t\leq {t}^*_{\epsilon}(b)}
{|\psi_{t}}|}{\sqrt{\epsilon}}}=\rho(b)\in ]0,+\infty[.
\end{eqnarray*}
for every $b\in \R$.
\end{lemma}

\begin{proof}
By continuity we have
\begin{eqnarray*}
\frac{\sup\limits_{{t}_{\epsilon}(b)\leq t\leq {t}^*_{\epsilon}(b)}
{|\psi_{t}}|}{\sqrt{\epsilon}}&=& \frac{
{|\psi_{\tilde{t}}}|}{\sqrt{\epsilon}}
\end{eqnarray*}
for some $\tilde{t}\in [{t}_{\epsilon}(b),{t}^*_{\epsilon}(b)]$.
Then, using the following relation and Lemma \ref{asin}, it becomes straightforward.
\begin{eqnarray*}
\frac{{|\psi_{\tilde{t}}}|}{\sqrt{\epsilon}}&=&e^{V^{\prime\prime}(0)\tilde{t}} |\psi_{\tilde{t}}|\frac{e^{-V^{\prime\prime}(0)\tilde{t}}}{\sqrt{\epsilon}}.
\end{eqnarray*}
\end{proof}

\section{Tools}\label{ap3}
\begin{lemma}\label{hopital}
$\lim\limits_{\epsilon\rightarrow 0}{\epsilon^{\alpha}\left(\ln\left(\frac{1}{\epsilon}\right)\right)^{\beta}}=0$ for every $\alpha>0$ and $\beta>0$.
\end{lemma}

\begin{lemma}\label{supinf}
Let $\{a_{\epsilon}\}_{\epsilon>0}\subset \R$ and $\{b_{\epsilon}\}_{\epsilon>0}\subset \R$ be sequences such that
$\lim\limits_{\epsilon\rightarrow 0}{b_{\epsilon}}=b\in \R$. Then
\begin{itemize}
\item[$i)$] $\limsup\limits_{\epsilon\rightarrow 0}{\left(a_{\epsilon}+b_{\epsilon}\right)}=\limsup\limits_{\epsilon\rightarrow 0}{a_{\epsilon}}+b$.
\item[$ii)$] $\liminf\limits_{\epsilon\rightarrow 0}{\left(a_{\epsilon}+b_{\epsilon}\right)}=\liminf\limits_{\epsilon\rightarrow 0}{a_{\epsilon}}+b$.
\end{itemize}
\end{lemma}

\begin{theorem}[Pinsker Inequality]\label{pinsker}
Let $\mu$ and $\nu$ be two probability measures define in the measurable space $\left(\Omega,\mathcal{F}\right)$. Then,
\begin{eqnarray*}
\norm{\mu-\nu}^2\leq 2\mathcal{H}\left(\mu\left|\right.\nu\right),
\end{eqnarray*}
where $\mathcal{H}\left(\mu\left|\right.\nu\right)$ is the Kullback information of $\mu$ respect to $\nu$ and it is defined as follows: if $\mu\ll\nu$ then take the Radon-Nikodym derivative $f=\frac{d\mu}{d\nu}$ and define
$\mathcal{H}\left(\mu\left|\right.\nu\right):=\int\limits_{\Omega}{f\ln(f)d\nu}$, in the case $\mu\not\ll\nu$ let us define $\mathcal{H}\left(\mu\left|\right.\nu\right):=+\infty$.
\end{theorem}
For details check \cite{BV}.

\markboth{}{References}
\bibliographystyle{amsplain}

\end{document}